\documentclass[11pt]{article}
\usepackage{marginnote}
\usepackage{graphbox} 
\usepackage{textcomp}

\usepackage{amsmath,amssymb, enumerate, fullpage}
\usepackage{graphicx,color}
\usepackage[T1]{fontenc}
\usepackage[utf8]{inputenc}
\usepackage{caption}
\usepackage[tight,footnotesize]{subfigure}
\usepackage[dvipsnames]{xcolor}
\usepackage{amsmath}
\usepackage{amsthm}
\usepackage{multirow}
\newtheorem{theorem}{Theorem}[section]

\newtheorem{definition}{Definition}
\newtheorem{lemma}[theorem]{Lemma}
\newtheorem{remark}{Remark}
\newtheorem{proposition}{Proposition}
\usepackage{verbatim}
\usepackage{epstopdf}
\usepackage{epsfig} 
\usepackage{bm}
\usepackage{rotating}
\usepackage{array, makecell, multirow}
\usepackage{tabularx,booktabs}
\usepackage{url} 
\usepackage[colorinlistoftodos]{todonotes}
\usepackage{graphicx,booktabs}

\usepackage{array,longtable,calc}
\newtheorem{assumption}{Assumption}

\newcommand{\tr}{^{\sf T}}
\newcommand{\vertiii}[1]{{\left\vert\kern-0.25ex\left\vert\kern-0.25ex\left\vert #1 
    \right\vert\kern-0.25ex\right\vert\kern-0.25ex\right\vert}}
\begin{document}
\title{Variable Metric Composite Proximal Alternating Linearized Minimization for
Nonconvex Nonsmooth Optimization}

\author{Maryam Yashtini }
\author{
 Maryam Yashtini 
\thanks{my496@georgetown.edu,
Georgetown University, 
Department of Mathematics and Statistics
327A St. Mary's Hall
37th and O Streets, N.W., Washington D.C. 20057
Phone: (202) 687-6214
Fax: (202) 687.6067
}
}\maketitle

\abstract{
In this paper we propose a proximal algorithm for minimizing an objective function of two block variables consisting of three terms: 1) a smooth function, 2) a nonsmooth function which is 
a composition between a strictly increasing, concave,
differentiable function and a convex nonsmooth function,
and 3) a smooth function which couples the two block 
variables. We propose a variable metric composite proximal alternating linearized minimization 
(CPALM) to solve this class of problems. Building on the powerful Kurdyka-\L ojasiewicz
property, we derive the convergence analysis and establish that each bounded 
sequence generated by CPALM globally converges to a critical point.
We demonstrate the CPALM method on parallel magnetic resonance image reconstruction problems.
The obtained numerical results shows the viability and effectiveness of the proposed method.
 }
 
 \vspace{.1in}

 \vspace{.1in}

\section{Introduction}
Consider the optimization problems of the form 
\begin{eqnarray}\label{originalprob}
\min_{x\in\mathbb R^n, y\in\mathbb R^m} F(x,y):=f(x)+g(y)+H(x,y)
\end{eqnarray}
where $f:\mathbb R^n\to (-\infty,+\infty]$ and $g:\mathbb R^m\to (-\infty,+\infty]$
are proper lower-semicontinuous and $H:\mathbb R^n\times \mathbb R^m\to \mathbb R$ is a smooth coupling function.
The well-known approach to solve (\ref{originalprob}) is to use the alternating minimization method,
that is, starting with some given initial point $(x^0,y^0)\in\mathbb R^n\times \mathbb R^m$ and 
to generate a sequence $\{(x^k,y^k)\}_{k\in\mathbb N}$ via the scheme:
\begin{eqnarray*}\label{GS}
\left\{
\begin{array}{lll}
x^{k+1}\in\displaystyle{\arg\min_{x\in\mathbb R^n}} \;\;\;F(x,y^k),\\[.1in]
y^{k+1}\in\displaystyle{\arg\min_{ y\in\mathbb R^m}} \;\;\;F(x^{k+1},y).
\end{array}
\right.
\end{eqnarray*}
The global convergence can be guaranteed if each minimization subproblem has a unique solution,
otherwise, the method may cycle infinitely without converging
\cite{Powell73,Zangwill69}. 
If the objective function $F$ is convex and continuously differentiable, and 
it is strictly convex over at least one variable, then every limit point of the generated sequence
minimizes $F$ \cite{BeckTetru13, BertsekasTsit97, Bertsekas97}. 
 Without the strict convexity assumption, one can modify the alternating minimization algorithm by 
 adding a proximal term:
\begin{eqnarray*}\label{GS+ProxTerm}
\left\{
\begin{array}{lll}
x^{k+1}\in\displaystyle{\arg\min_{x\in\mathbb R^n}} \;\;\;F(x,y^k)+\frac{\alpha_k}{2}\|x-x^k\|^2,\\[.1in]
y^{k+1}\in\displaystyle{\arg\min_{y\in\mathbb R^m}} \;\;\;F(x^{k+1},y)+\frac{\beta_k}{2}\|y-y^k\|^2,
\end{array}
\right.
\end{eqnarray*}
 where $\alpha_k$ and $\beta_k$ are positive real numbers. 
The subsequential convergence then can be proved in convex setting \cite{Auslender92,GriSci2000}.
When $F$ is nonconvex and nonsmooth, the situation becomes much harder. 
 In \cite{PALM14}, Bolte et al. considered an approximation of this approach via 
 proximal linearization of each subproblem.
 This yields the proximal alternating linearized minimization (PALM) algorithm:
\begin{eqnarray*}\label{GS+ProxTerm}
\left\{
\begin{array}{lll}
x^{k+1}\in\displaystyle{\arg\min_{ x\in\mathbb R^n}} \;\;\;f(x)+\langle \nabla_x H(x^k,y^k), x-x^k\rangle+\frac{\alpha_k}{2}\|x-x^k\|^2, \\[.1in]
y^{k+1}\in\displaystyle{\arg\min_{y\in\mathbb R^m}} \;\;\;g(y)+\langle \nabla_y H(x^k,y^k), y-y^k\rangle +\frac{\beta_k}{2}\|y-y^k\|^2.
\end{array}
\right.
\end{eqnarray*}
 Under the assumption that $F$ satisfies the Kurdyka-{\L}ojasiewicz property
 \cite{AttouchBolte09, PAM-KL-10, DM-KL-13, LojasiewicsSmoothsubana06,Clarkesubgradients07,Yashtini-multiblock,Yashtini-PLADMM},
 Bolte et al. \cite{PALM14} proved that each bounded sequence generated by PALM globally converges to a critical point, which is a stronger result than the subsequence convergence.
 On the other hand, the inertial scheme \cite{iPiano14,iPiasco}, starting from the so-called heavy ball method of 
 Polyak \cite{polyak64}, was recently proved to be very efficient in accelerating numerical methods, especially the first-order methods. Recently, there are increasing interests in studying inertial type algorithms, such as inertial forward-backward splitting methods for certain separable nonconvex optimization problems \cite{iPiano14} and for strongly convex optimization problems \cite{iPiasco}, inertial versions of the Douglas-Rachford operator splitting method \cite{Inertial-DR}.
 In particular, in \cite{GSiPALM} the authors consider a Gauss–Seidel type inertial proximal alternating linearized minimization (GiPALM) scheme for a class of nonconvex optimization problems.
 
In this work, we will aim to solve problem (\ref{originalprob}) where $g$ has the special composite structure 
\begin{eqnarray}\label{gstructure}
g(y)=(\phi\circ\psi)(y)
\end{eqnarray}
where $\phi:[0,+\infty]\to ]-\infty,+\infty]$ is concave, strictly increasing and differentiable function,
and $\psi:\mathbb R^m\to[0,+\infty[$ is convex, proper, lower-semicontinuous and Lipschitz continuous on its 
domain, such that $(\phi'\circ\psi)$ is Lipschitz-continuous on the domain of $\psi$, where $\phi'$ denotes the first 
derivative of $\phi$. 
To solve (\ref{originalprob})-(\ref{gstructure}) we consider a variable metric variant PALM-based algorithm 
as follows 
\begin{eqnarray}\label{VM-PALM}
\left\{
\begin{array}{lll}
x^{k+1}\in\displaystyle{\arg\min_{x\in\mathbb R^n}} \;\;\;
f(x)+\langle \nabla_x H(x^k,y^k), x-x^k\rangle+\frac{\alpha_k}{2}\|x-x^k\|_{A_k}^2,
\\[.1in]
y^{k+1}\in\displaystyle{\arg\min_{y\in\mathbb R^m}} \;\;\;
\phi\circ\psi(y)+\langle \nabla_y H(x^k,y^k), y-y^k\rangle +
\frac{\beta_k}{2}\|y-y^k\|^2_{B_k} 
\end{array}
\right.
\end{eqnarray}
where $A_k\in\mathbb R^{n\times n}$ and $B_k\in\mathbb R^{m\times m}$ are symmetric positive definite (SPD)
matrices, and can be considered appropriate pre-conditioners and possibly 
improve the stepsizes $\alpha_k$ and $\beta_k$ at each iteration.
The weighted norm associated with a matrix $A$ is defined as, for every $x\in\mathbb R^n $,
$\|x\|_A^2=\langle x, A x\rangle$.
When $A_k$ and $B_k$ are chosen to be equal to the identity matrices $I_n$ and $I_m$, respectively,
then the basic PALM is recovered. However, wiser choices of these matrices
can drastically accelerate the convergence of iterates, for instance see \cite{Chouzenoux14},
and also Section \ref{simulation}.
Due to the composite form of $g$, the $y$ subproblem might not be computable, either efficiently, or at all.
To overcome this difficulty, we propose to replace $g$ by a majorant function 
\[
q(\cdot,y^k):\mathbb R^m\to ]-\infty,+\infty]
\]
at $y^k$ at each iteration $k\in\mathbb N$ such that 
\begin{eqnarray}\label{g-majorant}
(\forall y\in\mathbb R^m)
\quad
\left\{
\begin{array}{l}
g(y)=(\phi\circ\psi)(y) \le q(y,y^k)\\[.1in]
g(y^k)=(\phi\circ\psi)(y^k)= q(y^k,y^k)
\end{array}
\right.
\end{eqnarray}
and obtained by taking the tangent of the concave differentiable function $\phi$ at $\psi(y^k)$:
\begin{eqnarray}\label{q}
q(y,y^k)= \phi \circ \psi(y^k)+
(\phi' \circ \psi)(y^k)\Big(\psi(y)-\psi(y^k)\Big).
\end{eqnarray}
The proposed variable metric composite proximal linearized alternating minimization (CPALM) is presented
in detail in Section \ref{proposedM}. We will prove the convergence of the sequence 
$\{(x^k,y^k)\}_{k\in\mathbb N}$ generated by the 
CPALM method to a critical point of $F$, using the K{\L} inequality.

The remainder of this paper is organized as follows.
In Section \ref{preliminaries} we introduce our notation and give useful definitions and preliminaries.
In Section \ref{proposedM} we propose the CPALM algorithm for solving (\ref{originalprob})-(\ref{gstructure})
and provide some necessary assumptions.
In Section \ref{Theory} we prove the global convergence of CPALM, and extend 
the CPALM method to solve more general problems with more than two blocks of variables.
Simulation results will be given in Section \ref{simulation}.
We conclude the paper in Section \ref{concludingremarks}.

\section{Notation and Preliminaries }\label{preliminaries}
Throughout this paper, we denote $\mathbb R$ as the real number set
while $\mathbb Z$ as the set of integers. 
The set $\mathbb R_+$ is the positive real number set
and $\mathbb Z_+$ is the set of positive integers.
 The domain  of $\Phi:\mathbb R^d\to ]-\infty, +\infty]$, denoted ${\rm dom}\; \Phi$, is defined by 
$
{\rm dom}\; \Phi:=\{x:\; \Phi(x)<+\infty\}.
$
We write $x^k$ is {\it $\Phi$-converging} to $x$,
and we write $x^k\xrightarrow{\Phi} x$, iff $x^k\to x$ and $\Phi(x^k)\to\Phi(x)$.
Given the matrix $X$, ${\rm Im}(X)$ denotes its image.
We denote by $I_n$ the $n\times n$ identity matrix for $n\in\mathbb Z_+$. The minimum, maximum,
and the smallest positive eigenvalues of the matrix $X\in\mathbb R^{n\times n}$ are denoted by 
$\lambda_{\min}(X)$, $\lambda_{\max}(X)$, $\lambda_+^X$, respectively.
The Euclidean scalar product of $\mathbb R^n$ and its corresponding 
norms are, respectively, denoted by $\langle \cdot, \cdot\rangle$
and $\|\cdot\|=\sqrt{\langle \cdot,\cdot\rangle}$.
 If  $n_1,\dots, n_p\in\mathbb Z_+$ and $p\in\mathbb Z_+$,
then for any $v:=(v_1,\dots,v_p)\in\mathbb R^{n_1}\times\mathbb R^{n_2}\times\dots\times\mathbb R^{n_p}$ and $v':=(v'_1,\dots,v'_p)\in \mathbb R^{n_1}\times\mathbb R^{n_2}\times\dots\times\mathbb R^{n_p}$ the Cartesian product and its norm are defined by
\begin{eqnarray}\label{norm}
\ll v,v'\gg =\sum_{i=1}^{p}\langle v_i,v_i'\rangle\quad\quad
\frac{1}{\sqrt{p}}  \sum_{i=1}^{p} \| v_i\| \le |||v||| = \sqrt{\sum_{i=1}^{p} \| v_i\|^2}\le  \sum_{i=1}^{p} \| v_i\|.
\end{eqnarray}
For the sequence $\{u^k\}_{k\ge 1}$, $\Delta u^{k}:=u^k-u^{k-1}$,
for all $k\ge 1$.  

Let $A\in\mathbb R^{n\times n}$ and $B\in\mathbb R^{n\times n}$ be two 
SPD matrices. We denote by $A\succeq B$ the loewner partial ordering on $\mathbb R^{n\times n}$,
defined as, for every $x\in\mathbb R^n$,
$x\tr A x\ge x\tr B x$.
The weighted norm associated with $A$ is defined as, for every $x\in\mathbb R^n $,
$\|x\|_A=\sqrt{\langle x, A x\rangle}$.

%

\subsection{Subdifferential and Critical Points}

\begin{definition}\label{def1}
Let $\Phi:\mathbb R^d\to (-\infty,\infty]$ be a proper and lower semicontinuous function.
\begin{itemize}
\item [(i)] For a given $\bar x\in\mathbb R^d$,
the {\it Fr\'echet subdifferential} of $\Phi$ at $\bar x$,
denoted by $\hat \partial \phi(\bar x)$, and is given by 
\begin{eqnarray}\label{F-sub}
\hat\partial\phi(\bar x)=
\big\{\hat s(\bar x)\in\mathbb R^d \Big| \;\;\lim_{y\to \bar x}\inf_{y\neq \bar x} 
\frac{\Phi(y)-\Phi(\bar x)-\langle \hat s(\bar x), y-\bar x\rangle}{\|y-\bar x\|}\ge 0\big\}.
\end{eqnarray}
If $x\notin {\rm dom} \Phi$, then $\hat \partial\Phi(\bar x):= \emptyset$.
\item [(ii)] The ({\it limiting}) subdifferential, or simply the subdifferential, 
of $\Phi$ at $\bar x\in\mathbb R^d$, denoted by $\partial\Phi(\bar x)$, 
and is defined by
\begin{eqnarray*}\label{lim-sub}
\partial\Phi(\bar x) &= &\big\{ s(\bar x)\in\mathbb R^d\Big|\;\;
\exists (x^k, s(x^k))\to (\bar x, s(\bar x))
\\
&&\quad\quad{\rm s.t.}\;\;\Phi(x^k)\to\Phi(\bar x) 
\;\; {\rm and}
\;\;  (\forall k\in\mathbb N) \;\;
\hat s(x^k)\in\hat\partial\phi(x^k)\big\}.
\end{eqnarray*}
Again, $\partial \Phi(\bar x):=\emptyset$ for $\bar x\notin {\rm dom} \Phi$,
and its domain is given by  ${\rm dom} \partial\Phi:=\{x\in\mathbb R^d: \partial\Phi(x)\neq \emptyset\}.$
\end{itemize}
\end{definition}

\begin{remark}\label{rem1}
\begin{itemize}
\item [(i)] An equivalent definition is given by 
\cite{RockafellarWets}: 
\[
\hat\partial\phi(\bar x)=
\big\{\hat s(\bar x)\in\mathbb R^d| (\forall y\in\mathbb R^d) \;\;\Phi(y)\ge \Phi(\bar x)+\langle \hat s(\bar x), y-\bar x\rangle+\o (|y-\bar x|)\}.
\]
\item [(ii)] Note that $\hat \partial\Phi(\bar x) \subset \partial \Phi(\bar x)$,
where the first set is convex and closed while the second 
one is closed (\cite{RockafellarWets}, Theorem 8.6).
 When $\Phi$ is convex the two sets coincide and
\[
\hat \partial\Phi (\bar x)=\partial\Phi (\bar x) =
\{s(\bar x)\in\mathbb R^d: \Phi(y)\ge \Phi(\bar x)
+\langle s(\bar x), y-\bar x\rangle\;\;\; \forall y\in\mathbb R^d\}.
\]
\item [(iii)] Let $\{(x^k,u^k)\}_{k\in\mathbb N}$ be a sequence in ${\rm graph}(\partial \Phi)$
that converges to $(\bar x, \bar u)$ as $k\to\infty$.
By the very definition of $\partial \Phi(x)$, if $\Phi(x^k)$ converges to $\Phi(\bar x)$
as $k\to\infty$, then $(\bar x, \bar u)\in {\rm graph} (\partial \Phi)$
\item [(iv)] A necessary condition for $x^*\in\mathbb R^d$ to be a minimizer of $\Phi$
is that $x^*$ is a critical point of $\Phi$,
i.e., $0\in\partial \Phi(x^*)$. If $\Phi$ is convex, this condition is also sufficient. 
We denote the set of {\it critical points} of $\Phi$ by ${\rm crit}\; \Phi$.
\item [(v)] The {\it lazy slope} of $\Phi$ at $\bar x$ is 
\[
\|\partial \Phi(\bar x)\|_-:=\inf \{\|s(\bar x)\|: {s(\bar x)\in\partial \Phi(\bar x)}\}
\]
if $\bar x\in{\rm dom}\;\partial\Phi$, and otherwise $+\infty$.
It follows from these definitions that if $x^k\xrightarrow{\Phi} x$ with 
$\lim\inf_{k\to\infty}\|\partial \Phi(x)\|_-=0$, then $x$ is a critical point.

\end{itemize}
\end{remark}

\begin{proposition}{\bf (Subdifferential property)}
Assume that $F:\mathbb R^n\times\mathbb R^m\to]-\infty,+\infty]$ is continuously
differentiable, then for all $(x,y)\in \mathbb R^n\times\mathbb R^m$ we have
\[
\partial F(x, y)=\big(\partial_x F(x,y), \partial_y F(x,y)\big),
\]
where $\partial_x F$ and $\partial_y F$ are respectively the differential of the function 
$F(\cdot,y)$ when $y\in\mathbb R^m$ is fixed, and $F(x,\cdot)$ when $x\in\mathbb R^n$ is fixed.

\end{proposition}
\begin{proposition}{\bf (Chain Rule)}\cite{RockafellarWets}
\begin{itemize}
\item [(i)] Let $h:\mathbb R^n\to\mathbb R$ be a differentiable function and $g:\mathbb R^n\to ]-\infty,+\infty]$,
then we have $\partial(h+g)=\nabla h+\partial g$.
\item [(ii)]
Let, for every $p=1,2,\dots, P$,
$\psi_p:\mathbb R^n\to (-\infty,+\infty]$ be a convex,
proper, and lower-semicontinuous functions, then  
we have $\partial (\sum_{p=1}^P \psi_p)=\sum_{p=1}^P \partial \psi_p$. 
\end{itemize}
\end{proposition}

\begin{lemma}\label{desclemma} \cite{Yashtini-multiblock}
If $\Phi:\mathbb R^d\to\mathbb R$ is Fr\'echet differentiable and 
its gradient is Lipschitz continuous with constant $L_{\Phi}>0$. Then 
for every $u,v\in\mathbb R^d$ and every $\xi\in[u,v]=\{(1-t)u+tv:\;t\in[0,1]\}$
it holds 
\begin{eqnarray}\label{Lip}
\Phi(v)\le \Phi(u)+\langle \nabla \Phi(\xi),v-u\rangle +\frac{L_{\Phi}}{2}\|v-u\|^2.
\end{eqnarray}
Moreover, if $\Phi$ is bounded from below,  by setting
 $\xi=u$ and $v=u-\delta \nabla\Phi(u)$ in (\ref{Lip}), it is easy
 to show that the term
$\Phi(v)-(\delta-\frac{L_{\Phi} \delta^2}{2})\|\nabla \Phi(v)\|^2$
is bounded from below for every $\delta>0$.
\end{lemma}

\begin{proposition}
\cite{RepettiW21} Let $\psi:\mathbb R^n\to [0,+\infty]$ be a proper function which is
continuous on its domain, and let $\phi:[0,+\infty]\to (-\infty,+\infty]$
be a concave, strictly increasing and differentiable function.
We further assume that $(\phi'\circ\psi)$ is continuous on its domain.
Then, for every $\bar x\in\mathbb R^n$, we
have
\begin{eqnarray*}
\partial(\phi\circ \psi)(\bar y)=(\phi'\circ \psi)(\bar y)\partial\psi(\bar y)
\end{eqnarray*}
\end{proposition}

\subsection{\bf The Kurdyka-{\L}ojasiewicz Properties}\label{sub:KL}
The Kurdyka-{\L}ojasiewicz (KL) property plays a central role in our analysis. 
We refer interested readers to 
\cite{AttouchBolte09, PAM-KL-10, DM-KL-13, LojasiewicsSmoothsubana06,Clarkesubgradients07,CharacLineq10,PALM14,FrankelGuillaumePey15,Yashtini-multiblock,Yashtini-PLADMM}
for more properties of KL functions and illustrating examples. 
Below, we recall the essential elements. 

Let $\eta>0$. We denote by $\Gamma_{\eta}$ the class of concave and 
continuous functions $\psi:[0,\eta[\to [0,+\infty[$ which satisfy the following conditions
\begin{itemize}
\item [(i)] $\psi(0)=0$;
\item [(ii)] $\psi$ is continuously differentiable on $]0,\eta[$;
\item [(iii)] for all $t\in]0,\eta[$: $\psi'(t)>0$.
\end{itemize}

\begin{definition} {\bf(KL property)}
Let $\Phi:\mathbb R^d\to ]-\infty,+\infty]$
be proper and lower semi-continuous. 
\begin{itemize}
\item [(i)] The function $\Phi$ is said to have the KL 
 property at $\bar x \in {\rm dom} \;\partial \Phi:=\{x\in\mathbb R^d:\partial \Phi(x)\neq\emptyset\}$ 
if there exists $\epsilon>0$, $\eta>0$, and a {\it desingularising function} $\psi\in\Gamma_{\eta}$
such that for all $x$ in the {\it strict local upper level} set 
 \begin{eqnarray}\label{Gamma}
 \Gamma_{\eta}(\bar x, \epsilon)
 =\{x\in\mathbb R^d: \|x-\bar x\|<\epsilon\;\; {\rm and} \;\;
 \Phi(\bar x)<\Phi(x)<\Phi(\bar x)+\eta\}.
 \end{eqnarray}
the following KL inequality holds 
 \begin{eqnarray}\label{KLine}
\psi'\big(\Phi(x)-\Phi(\bar x)\big) \| \partial \Phi(x)\|_- \ge 1.
\end{eqnarray}
\item [(ii)] If $\Phi$ satisfy the KL 
 property at each point of ${\rm dom} \partial \Phi$ then
 $\Phi$ is called a {K\L} function.
\end{itemize}

 \end{definition}
 
 \begin{lemma}{\bf (Uniformized KL property)}\label{UniKL} \cite{PALM14}
Let $\Omega\subset \mathbb R^d$ be a compact set 
and let $\Phi:\mathbb R^d\to ]-\infty,+\infty]$ be a proper and 
lower semicontinuous function.
Assume that $\Phi$ is constant on $\Omega$ and satisfies the 
KL property at each point of $\Omega$.
Then, there exists $\epsilon>0$, $\eta>0$, and 
$\psi\in\Gamma_{\eta}$ such that for all $\bar x\in\Omega$
and all $x$ belongs to the following intersection
 \begin{eqnarray}\label{GammaSet}
\Big\{ x\in\mathbb R^d: {\rm dist} (x,\Omega)<\epsilon\Big\} \cap [\Phi(\bar x)<\Phi(x)<\Phi(\bar x)+\eta]
 \end{eqnarray}
 one has $\psi'\big(\Phi(x)-\Phi(\bar x)\big) \| \partial \Phi(x)\|_- \ge 1.$
 
 \end{lemma}
 
 \begin{remark}\label{rem2}
 We make the following remarks:
 \begin{itemize}
 \item [(i)]  When $\Phi$ is of class $C^1$, 
 (\ref{KLine}) becomes $\|\nabla (\psi\circ \Phi)(x)\|\ge 1$.
  This means that the more $\Phi$ is flat around its critical points,
 the more $\psi$ has to be steep around $0$, and this justifies the term
 ``desingularising". 
\item [(ii)]
The growth of $\psi$ has a direct impact on the convergence rate of the algorithm. 
If  $\psi(s)=\bar c s^{1-\theta}$ for some $\bar c>0$ and $\theta\in [0,1)$, then
 the  KL  inequality (\ref{KLine}) becomes
\begin{eqnarray}\label{KLexp}
\big(\Phi(x)-\Phi(\bar x)\big)^{\theta}\le c \| \partial \Phi(x)\|_- 
\end{eqnarray}
for any  $x\in \Gamma_{\eta}(\bar x, \epsilon)$ where $c=(1-\theta)\bar c$. 
In this case, we say $\Phi$ has the KL property at $\bar x$ with an exponent $\theta$.
This asserts that $(\Phi(x)-\Phi(\bar x))^{\theta}/\|\partial \Phi(x)\|_- $ remains
bounded around $\bar x$.
 \end{itemize}
 \end{remark}


\section{Proposed Method }\label{proposedM}
In this section, we detail the structure of the proposed algorithm, we provide the necessary 
assumptions for the convergence.

The algorithm that we propose to solve (\ref{originalprob})-(\ref{gstructure})
is based on the PALM \cite{PALM14}, 
which in a variable metric format it is expressed in (\ref{VM-PALM}). 
Due to the composition structure of $g$ given in (\ref{gstructure}), the solution of 
$y$ subproblem in (\ref{VM-PALM}) might not be computable,
either efficiently or at all. For this reason,
we propose to replace $g$ by a majorant function $q(\cdot,y^k)$ at $y^k$
at each iteration $k\in\mathbb N$ satisfying in (\ref{g-majorant})  and
 (\ref{q}). This leads us to 
\begin{eqnarray}\label{VM-PALM-majorant}
\left\{
\begin{array}{lll}
x^{k+1}\in\displaystyle{\arg\min_{x\in\mathbb R^n}} \;\;\;
f(x)+\langle \nabla_x H(x^k,y^k), x-x^k\rangle+\frac{\alpha_k}{2}\|x-x^k\|_{A_k}^2,
\\[.1in]
y^{k+1}\in\displaystyle{\arg\min_{y\in\mathbb R^m}} \;\;\;
q(y,y^k)+\langle \nabla_y H(x^k,y^k), y-y^k\rangle +
\frac{\beta_k}{2}\|y-y^k\|^2_{B_k} 
\end{array}
\right.
\end{eqnarray}
By the fact that 
\[
q(\cdot,y^k)=(\phi' \circ \psi)(y^k) \psi(\cdot)+C_k
\]
 where $C_k\in\mathbb R$, without loss of generality, we can replace 
 $q(y,y^k)$ by  $(\phi' \circ \psi)(y^k) \psi(y)$ in the $y$ subproblem
 to obtain
 \begin{eqnarray*}
y^{k+1}\in\displaystyle{\arg\min_{y\in\mathbb R^m}}
(\phi' \circ \psi)(y^k) \psi(y)+
\langle \nabla_y H(x^k,y^k), y-y^k\rangle +\frac{\beta_k}{2}\|y-y^k\|_{B_k}^2.
\end{eqnarray*}
The proposed CPALM method to solve problem (\ref{originalprob})-(\ref{gstructure}) is given below.

\vspace{.2in}

{\bf CPALM: Composite Proximal Alternating Linearized Minimization }
\begin{itemize}
\item[1.] Initialization: choose a starting point $(x^0,y^0)\in\mathbb R^n\times \mathbb R^m$.
\item[2.] For each $k=0,1,\dots$ generate a sequence $\{(x^k,y^k)\}_{k\in\mathbb N}$ as follows
\begin{itemize}
\item[2.1.] Take $\rho_1=\max\{1,\frac{1}{\lambda_{\min}(A_k)}\}$, $\gamma_1>1$, set $\alpha_k=\gamma_1\rho_1L_1(y^k)$ and compute
\begin{eqnarray}\label{x-CPALM}
x^{k+1}\in\displaystyle{\arg\min_{x\in\mathbb R^n}} \;\;\;\Big\{f(x)+\langle \nabla_x H(x^k,y^k), x-x^k\rangle+\frac{\alpha_k}{2}\|x-x^k\|_{A_k}^2 \Big\}
\end{eqnarray}
\item[2.2.] Take 
$\rho_2=\max\{1,\frac{1}{\lambda_{\min} (B_k)}\}$, $\gamma_2>1$, set $\beta_k=\gamma_2\rho_2L_2(x^k)$
and compute
\begin{eqnarray}\label{y-CPALM}
y^{k+1}\in\displaystyle{\arg\min_{y\in\mathbb R^m}}
\Big\{
(\phi'\circ\psi)(y^k) \psi(y)+\langle \nabla_y H(x^k,y^k), y-y^k\rangle +
\frac{\beta_k}{2}\|y-y^k\|_{B_k}^2
\Big\}.
\end{eqnarray}

\end{itemize}
\end{itemize}
\vspace{.2in}

 \begin{assumption}\label{ass1}
 The following assumptions are considered on the functions $f, g$, and $H$:
\item [(i)] $f:\mathbb R^n \to ]-\infty,+\infty]$ and 
$g:\mathbb R^m\to ]-\infty, +\infty]$ are proper  lower semicontinuous 
functions such that $\inf_{\mathbb R^n} f>-\infty$, and $\inf_{\mathbb R^m} g>-\infty$.
 \item [(ii)] $H:\mathbb R^n\times \mathbb R^m$ is differentiable
 and $\inf_{\mathbb R^n\times \mathbb R^m} F>-\infty$
 \item [(iii)] For any fixed $y$ the function $x\rightarrow H(x,y)$ is  a $C_{L_1(y)}^{1,1}$,
 namely, the partial gradient $\nabla_x H(x,y)$ is globally Lipschitz with moduli
 $L_1(y)$, that is 
 \[
\Big \|\nabla_x H(x_1,y)-\nabla_x H(x_2,y)\Big \|\le L_1(y)\Big\|x_1-x_2\Big\|,\quad\forall x_1,x_2\in\mathbb R^n.
 \]
Likewise, for any fixed $x$ the function $y\to H(x,y)$ is assumed to be $C^{1,1}_{L_2(x)}$.
\item [(iii)] For $i=1,2$ there exists $\lambda_i^-,\lambda_i^+>0$ such that
\begin{eqnarray}
\inf\{L_1(y^k):k\in\mathbb N\}\ge \lambda_1^-& \quad &
\inf\{L_2(x^k):k\in\mathbb N\}\ge \lambda_2^-
\label{inf}
\\
\sup\{L_1(y^k):k\in\mathbb N\}\le \lambda_1^+
& \quad &
\sup\{L_2(x^k):k\in\mathbb N\}\le \lambda_2^+
\label{sup}
\end{eqnarray}
\item [(iv)] $\nabla H$ is Lipschitz continuous on bounded subsets
of $\mathbb R^n\times\mathbb R^m$.
In other words,
for each bounded subsets $B_1\times B_2$ of $\mathbb R^n\times\mathbb R^m$
there exists $M>0$ such that for all $(x_i,y_i)\in B_1\times B_2$, $i=1,2$:
\begin{eqnarray}\label{ineq:LipH}
\Big\|\nabla_x H(x_1,y_1)-\nabla_x H(x_2,y_2), 
\nabla_y H(x_1,y_1)-\nabla_y H(x_2,y_2)
\Big\|
\le M \Big\|\big(x_1-x_2, y_1-y_2\big)\Big\|
\end{eqnarray}
 \end{assumption}
 
 \begin{assumption}\label{ass2}
The following assumptions are considered on the functions $\phi$ and $\psi$:
\begin{itemize}
\item [(i)] The function $\phi:[0,+\infty]\to ]-\infty,+\infty]$ is concave and 
strictly increasing (i.e. $\phi'(u)>0$ for every $u\in[0,+\infty].$)
Moreover, it is differentiable on $[0,+\infty[$.
\item [(ii)] The function $\psi:\mathbb R^m\to[0,+\infty]$ is convex, proper, lower-semicontinuous. Moreover,
it is Lipschitz continuous on its domain. 
\item [(iii)] The function $\phi'$ is locally Lipschitz continuous on its domain.
\end{itemize}
 \end{assumption}
 
  \begin{assumption}\label{ass3}
 The following assumptions are made on matrices $\{A_k\}_{k\in\mathbb N}$ and $\{B_k\}_{k\in\mathbb N}$:
 \begin{itemize}
\item [(i)] The matrices $A_k\in\mathbb R^{n\times n}$ and $B_k\in\mathbb R^{m\times m}$ are symmetric positive definite (SPD) matrices.
\item [(ii)] We let $\underline a=\min\{ \lambda_{\min}(A_k):k\in\mathbb N\}$ and
$\underline b=\min\{ \lambda_{\min}(B_k):k\in\mathbb N\}$.
\item [(iii)] $\bar a= \max\{\|A_k\|:k\in\mathbb N\}$ and $\bar b= \max\{\|B_k\|:k\in\mathbb N\}$.
\end{itemize}
  \end{assumption}
\begin{remark}\label{rem3}
We make the following remarks:
\begin{itemize}
\item [(i)] According to Assumption \ref{ass2} (ii), 
there exists $\nu>0$ such that $\|r(y)\|\le \nu$  for every $r(y)\in\partial\psi(y)$ with $y\in dom\; \psi$.
\item [(ii)] According to Assumption \ref{ass2} (i)-(ii),
$\phi'\circ\psi$ is continuous on its domain.
\item [(iii)]Assumption \ref{ass2}(iii) holds if and only if 
the function $\phi'\circ\psi$ is Lipschitz continuous on every compact subset of $\mathbb R^m.$
Thus, under Assumption \ref{ass2}(iii), there exists $\mu>0$ such that for every $y',y''\in dom g$,
$\|(\phi'\circ \psi)(y')- (\phi'\circ\psi)(y'')\|\le \mu \|y'-y''\|$.
\end{itemize}

\end{remark} 
  
\section{Theoretical Analysis}\label{Theory}
We begin this section with some lemmas necessary to prove the convergence of the proposed
CPALM method in Theorem \ref{conv}.
  
\subsection{Basic Properties}
\begin{lemma}\label{lemma:SuffDec}
{\bf (Sufficient decrease property)}
Let $h:\mathbb R^d\to\mathbb R$ be a continuously differentiable function with gradient $\nabla h$
assumed $L_h$-Lipschitz continuous and let $\sigma:\mathbb R^d\to]-\infty,+\infty]$ be a proper and 
lower semicontinuous function with $\inf_{\mathbb R^d} \sigma >-\infty$. 
Then for any $u\in dom\; \sigma$ and $u^+\in\mathbb R^d$ defined by
\begin{eqnarray}\label{suf-sub}
u^+\in\arg\min_{v\in\mathbb R^d} \Big\{ \sigma(v)+\langle v-u, \nabla h(u)\rangle +\frac{t}{2}\|v-u\|^2_A\Big\},
\end{eqnarray}
where $A$ is a SPD matrix and $t>0$ 
we have 
\begin{eqnarray*}
\sigma(u^+)+h(u^+)\le \sigma(u)+h(u)+\frac{L_h}{2}\|u^+-u\|^2-\frac{t}{2}\|u^+-u\|^2_A.
\end{eqnarray*}

\end{lemma}
\begin{proof}
First, note that (\ref{suf-sub}) can be equivalently expressed as 
\begin{eqnarray}\label{proxdef}
u^+\in {\rm prox}^{\sigma}_{t A}\Big(u-\frac 1t A^{-1} \nabla h(u)\Big)
\end{eqnarray}
and by the fact that $\sigma$ is proper lower-semicontinuous and bounded from
below, the subproblem (\ref{suf-sub}) is well-defined (Please see Proposition 2 in \cite{PALM14}).  
Since $u^+$ is the minimizer, for any $v\in\mathbb R^d$ we have 
\[
\sigma(u^+)+\langle u^+-u, \nabla h(u)\rangle +\frac{t}{2}\|u^+-u\|^2_A\le \sigma(v)+\langle v-u, \nabla h(u)\rangle +\frac{t}{2}\|v-u\|^2_A.
\]
We now let $v=u$ on the right hand side to obtain
\begin{eqnarray}\label{eq:1}
\sigma(u^+)+\langle u^+-u, \nabla h(u)\rangle +\frac{t}{2}\|u^+-u\|^2_A\le \sigma(u)
\end{eqnarray}
Exploiting Lemma \ref{desclemma} with $\Phi=h$, $v=u^+$, and $\xi=u$ we have
\begin{eqnarray}\label{eq:hh}
h(u^+)\le h(u)+\langle u^+-u,\nabla h(u)\rangle +\frac{L_h}{2}\|u^+-u\|^2. 
\end{eqnarray}
We add $\sigma (u^+)$ to both sides of  (\ref{eq:hh}) then using (\ref{eq:1}),
it yields 
\begin{eqnarray*}\label{eq:1+}
\sigma(u^+)+h(u^+)&\le& h(u)+\langle u^+-u,\nabla h(u)\rangle +\frac{L_h}{2}\|u^+-u\|^2+\sigma(u^+)\\
&\le & \sigma(u)+ h(u)+ \frac{L_h}{2}\|u^+-u\|^2- \frac{t}{2}\|u^+-u\|^2_A.
\end{eqnarray*}
This completes the proof. $\square$

\end{proof}

\begin{lemma}\label{CnvProp}
{\bf (Sufficient decrease property)}
Suppose that Assumptions \ref{ass1} and \ref{ass2} hold. Let 
\[
\big\{z^k:=(x^k,y^k)\big\}_{k\in\mathbb N}
\]
be a sequence generated by the CPALM method. Then the following assertions hold.
\begin{itemize}
\item [(i)] The sequence $\{F(z^k)\}_{k\in\mathbb N}$ is nonincreasing and in particular
\begin{eqnarray}\label{nondecF}
F(z^k)-F(z^{k+1}) \ge \delta |||z^{k+1}-z^k|||^2
\end{eqnarray}
for some $\delta>0$.
\item [(ii)] The sequence $\{z^k\}$ has a finite length, that is, 
\begin{eqnarray}\label{finiteseq}
\sum_{k=0}^{\infty}|||z^{k+1}-z^k|||^2
= \sum_{k=0}^{\infty} \|x^{k+1}-x^k\|^2 +\|y^{k+1}-y^k\|^2<\infty,
 \end{eqnarray}
 and hence
$
 \lim_{k\to\infty} |||z^{k+1}-z^k|||=0.
$
\end{itemize}
\end{lemma}
\begin{proof}
(i) Fix $k\ge 0$. By the Assumption, the functions $x\rightarrow H(x,y)$ 
and $y\rightarrow H(x,y)$ are differentiable and have a Lipschitz continuous gradients with 
modulis $L_1(y)$ and $L_2(x)$, respectively. 
Applying the Lemma \ref{lemma:SuffDec} with $h(\cdot):=H(\cdot,y^k)$ and $\sigma=f$, $t=\alpha_k$, and $A=A_k$,
and invoking $\rho_1=\max\{1, 1/\lambda_{\min}(A_k)$,
$\gamma_1>1$, and $\alpha_k=\gamma_1\rho_1 L_1(y^k)$ 
we obtain 
\begin{eqnarray*}\label{c1}
f(x^{k+1})+H(x^{k+1},y^k)&\le& f(x^k)+H(x^k,y^k)+\frac{L_1(y^k)}{2}\|x^{k+1}-x^k\|^2-\frac{\alpha_k}{2}\|x^{k+1}-x^k\|^2_{A_k} 
\\
&\le& f(x^k)+H(x^k,y^k)-\frac 12 \Big(\alpha_k-\rho_1L_1(y^k)\Big) \|x^{k+1}-x^k\|
_{A_k}^2 \\
&\le& f(x^k)+H(x^k,y^k)-\frac 12 L_1(y^k) \rho_1(\gamma_1-1) \|x^{k+1}-x^k\|.
_{A_k}^2 \\
\end{eqnarray*}

Next, we use the Lemma \ref{lemma:SuffDec} with 
$h(\cdot):=H(x^k,\cdot)$ and $\sigma=q(\cdot,y^k)$, $t=\beta_k$, and $A=B_k$
together with $\rho_2=\max\{1,1/\lambda_{\min} (B_k)\}$, $\gamma_2>1$, 
$\beta_k=\gamma_2\rho_2L_2(x^k)$ as well as (\ref{g-majorant})
we get
\begin{eqnarray*}\label{c2}
g(y^{k+1})+H(x^{k+1},y^{k+1})&\le &q(y^{k+1},y^k)+H(x^{k+1},y^{k+1})\nonumber \\
&\le& q(y^{k},y^k)+H(x^{k+1},y^k)+
\frac{L_2(x^k)}{2}\|y^{k+1}-y^k\|^2-\frac{\beta_k}{2}\|y^{k+1}-y^k\|^2_{B_k} \nonumber\\
&\le &q(y^{k},y^k)+H(x^{k+1},y^k)-\frac 12\Big(\beta_k-\rho_2L_2(x^k)\Big)\|y^{k+1}-y^k\|^2_{B_k}\nonumber \\
&\le &q(y^{k},y^k)+H(x^{k+1},y^k)-\frac 1 2 L_2(x^k) \rho_2(\gamma_2-1) \|y^{k+1}-y^k\|
_{B_k}^2
\\
&= &g(y^{k})+H(x^{k+1},y^k)-\frac 1 2 L_2(x^k) \rho_2(\gamma_2-1) \|y^{k+1}-y^k\|.
_{B_k}^2
\end{eqnarray*}

We then add the above two inequalities, and this leads us to 
\begin{eqnarray}\label{gg}
\begin{array}{l}
F(x^k, y^k)- F(x^{k+1}, y^{k+1}) \\[.1in]
\quad\quad= f(x^k)+g(y^k)+H(x^k,y^k)- f(x^{k+1})-g(y^{k+1})-H(x^{k+1},y^{k+1})\\[.1in]
\quad\quad \ge
 \frac 12 L_1(y^k) \rho_1\big(\gamma_1-1\big) \|x^{k+1}-x^k\|
_{A_k}^2 
+\frac 1 2 L_2(x^k) \rho_2\big(\gamma_2-1\big) \|y^{k+1}-y^k\|_{B_k}^2.
\end{array}
\end{eqnarray}
This result shows that the sequence $\{F(z^k)\}_{k\in\mathbb N}$ is nonincreasing,
and since $F$ is assumed to be bounded from below by Assumption \ref{ass1}(i),(ii),
it converges, let say to $\underline F\in\mathbb R$. 

By Assumption \ref{ass1} (iii), $L_1(y^k)\ge \lambda_1^->0$ and $L_2(x^k)\ge \lambda_2^->0$.
Let also assume $\underline a=\min\{ \lambda_{\min}(A_k):k\in\mathbb N\}$ and
$\underline b=\min\{ \lambda_{\min}(B_k):k\in\mathbb N\}$, then we get 
\begin{eqnarray}\label{cc}
\begin{array}{l}
 \frac 12 L_1(y^k) \rho_1(\gamma_1-1) \|x^{k+1}-x^k\|
_{A_k}^2 
+\frac 1 2 L_2(x^k) \rho_2(\gamma_2-1) \|y^{k+1}-y^k\|_{B_k}^2 
\\[.1in]
\quad\quad\ge  \frac 12 \lambda_1^- \rho_1 \underline a (\gamma_1-1)  \|x^{k+1}-x^k\|^2 
+\frac 1 2 \lambda_2^- \rho_2\underline b(\gamma_2-1) \|y^{k+1}-y^k\|^2 
\\[.1in]
\quad\quad\ge \frac{\delta}{2} \|x^{k+1}-x^k\|^2 +\frac{\delta}{2}\|y^{k+1}-y^k\|^2
= \frac{\delta}{2} |||z^{k+1}-z^k|||^2
\end{array}
\end{eqnarray}

where 
$\delta=\min \{ \lambda_1^- \rho_1  \underline a(\gamma_1-1), \lambda_2^- \rho_2\underline b(\gamma_2-1) \}$.
By combining (\ref{gg}) and (\ref{cc}) we obtain 
\begin{eqnarray}\label{Fdec}
F(z^k)- F(z^{k+1}) \ge \frac{\delta}{2} \|z^{k+1}-z^k\|^2,
\end{eqnarray}
which assertion (i) is proved.

(ii) Let $K$ be a fixed positive integer. Summing up (\ref{Fdec}) from $k=0$ to $K-1$ gives
\begin{eqnarray*}
\sum_{k=0}^{K-1} \|x^{k+1}-x^k\|^2 +\|y^{k+1}-y^k\|^2=
\sum_{k=0}^{K-1}  |||z^{k+1}-z^k|||^2 
\le  \frac{2}{\delta} \Big(F(z^0)-F(z^K)\Big)
\le  \frac{2}{\delta} \Big(F(z^0)-\underline{F}\Big).
\end{eqnarray*}
We let $K$ to approach to infinity to obtain the desired result (ii). $\square$
\end{proof}

\subsection{Approaching the Set of Critical Points}
\begin{lemma}\label{subgbound}
{\bf (A subgradient bound)}
Suppose that Assumptions \ref{ass1} and \ref{ass2} hold. Let $\{z^k\}_{k\in\mathbb N}$ 
be a sequence generated by the CPALM algorithm which is assumed to be bounded. 
For each $k\in\mathbb Z_+$, we define
\begin{eqnarray}
d_x^{k+1}&:= &\nabla_xH(x^{k+1},y^{k+1})-\nabla_xH(x^k,y^k)+\alpha_kA_k (x^{k}-x^{k+1}), \label{dx}\\[.1in]
d_y^{k+1}&:=& \Big((\phi'\circ\psi)(y^{k+1})-(\phi'\circ\psi)(y^{k})\Big)w^{k+1} +\beta_kB_k(y^k-y^{k+1}) \nonumber\\
&&+\nabla_yH(x^{k+1},y^{k+1})   - \nabla_yH(x^{k+1},y^{k}) \label{dy} 
\end{eqnarray}
where $w^{k+1}\in \partial \psi(y^{k+1})$.
Then $(d_x^{k},d_y^{k})\in\partial F(x^k,y^k)$, and there exists $M>0$ and $\nu>0$ such that 
\begin{eqnarray}\label{subgradnorm}
\vertiii{(d_x^{k+1}, d_y^{k+1})}\le \|d_x^{k+1}\|+\|d_y^{k+1}\|\le (2M+\mu \nu+3\xi)\vertiii{z^{k+1}-z^{k}}.
\end{eqnarray}
where 
\[
\xi=\max\{\bar a \gamma_1\rho_1\lambda_1^+, \bar b \gamma_2\rho_2 \lambda_2^+\}.
\]
\end{lemma}

\begin{proof}
We note that 
\begin{eqnarray*}
\partial F(x,y)=\Big(\partial_x F(x,y), \partial_y F(x,y)\Big)=
\Big(\partial f(x)+\nabla_xH(x,y), \partial g(y)+\nabla_y H(x,y)\Big).
\end{eqnarray*}
Let $k\in\mathbb Z_+$ be a positive integer, and 
let $v^k\in \partial f(x^k)$. The optimality condition of the $x$-minimization subproblem 
(\ref{x-CPALM}) of CPALM method is given by
\begin{eqnarray}\label{x-opt}
v^{k+1}+ \nabla_xH(x^k,y^k)+\alpha_kA_k (x^{k+1}-x^k)=0
\end{eqnarray}
Therefore, it is easy to see that 
\begin{eqnarray}\label{vx}
v^{k+1}+ \nabla_xH(x^{k+1},y^{k+1})\in \partial_x F(x^{k+1},y^{k+1}).
\end{eqnarray}
By (\ref{x-opt}) and (\ref{vx}), then we have
\begin{eqnarray}\label{dxx}
d_x^{k+1}:= \nabla_xH(x^{k+1},y^{k+1})-\nabla_xH(x^k,y^k)+\alpha_kA_k (x^{k}-x^{k+1})\in \partial_x F(x^{k+1},y^{k+1}).
\end{eqnarray}

Similarly, the optimality condition of the $y$-minimization subproblem (\ref{y-CPALM})
of CPALM method can be expressed by
\[
(\phi\circ\psi)(y^k)w^{k+1}+\nabla_yH(x^{k+1},y^k)+\beta_k B_k (y^{k+1}-y^k)=0
\]
where $w^{k+1}\in \partial \psi(y^{k+1})$. 
We note that 
$
\partial g(y)=\partial (\phi\circ\psi)(y)=(\phi'\circ\psi)(y)\partial \psi(y).$
Thus we have $(\phi'\circ\psi)(y^{k+1}) w^{k+1} \in \partial g(y^{k+1})$.
Thus 
\begin{eqnarray}\label{dyy}
d_y^{k+1}&:=&\Big((\phi'\circ\psi)(y^{k+1})-(\phi'\circ\psi)(y^{k})\Big)w^{k+1}\nonumber\\
&&+
 \nabla_yH(x^{k+1},y^{k+1}) - \nabla_yH(x^{k+1},y^{k})
 +\beta_kB_k(y^k-y^{k+1})
\in \partial_y F(x^{k+1},y^{k+1})\nonumber\\
\end{eqnarray}
Therefore $(d_x^{k+1},d_y^{k+1})\in \partial F(x^{k+1},y^{k+1})$.

Next, we obtain the norms of $d_x^{k+1}$ and $d_y^{k+1}$.
Since $\nabla H$ is Lipschitz continuous on bounded subsets of $\mathbb R^n\times \mathbb R^m$
and since the 
sequence $\{z^k\}_{k\in\mathbb N}$ assumed to be bounded, there exists $M>0$ such that

\begin{eqnarray*}
\|d_x^{k+1}\|&\le& \big\| \nabla_xH(x^{k+1},y^{k+1})-\nabla_xH(x^k,y^k)\big\|+\alpha_k\|A_k\| \|x^{k+1}-x^{k}\|\\[.1in]
&\le& M\Big(\|x^{k+1}-x^k\|+ \|y^{k+1}-y^k\|\Big)+\alpha_k\|A_k\| \|x^{k+1}-x^{k}\|\\[.1in]
&=& \Big(M+\alpha_k\|A_k\|\Big)\|x^{k+1}-x^k\|+M \|y^{k+1}-y^k\|.
\end{eqnarray*}
By Assumption \ref{ass3}(iii), the matrix norms $\|A_k\|$ are bounded above by $\bar a$ for all $k\in\mathbb N$.
By the CPALM method we also have $\alpha_k=\gamma_1\rho_1L_1(y^k)\le \gamma_1\rho_1\lambda_1^+$,
thus
\begin{eqnarray*}
\|d_x^{k+1}\| &\le& (M+\bar a \gamma_1\rho_1\lambda_1^+)\|x^{k+1}-x^k\|+M \|y^{k+1}-y^k\|\\
&\le& (2M+\bar a \gamma_1\rho_1\lambda_1^+)  \vertiii{z^{k+1}-z^k}\\
&\le & (2M+\xi) \vertiii{z^{k+1}-z^k}.
\end{eqnarray*}
Next, we use Assumption \ref{ass1} and \ref{ass3}, and Remark \ref{rem3} together with 
$\beta_k=\gamma_2\rho_2L_2(x^k)\le \gamma_2\rho_2\lambda_2^+$ to obtain 
\begin{eqnarray*}
\|d_y^{k+1}\|
&\le& \big\|(\phi'\circ\psi)(y^{k+1})-(\phi'\circ\psi)(y^{k})\big\| \|w^{k+1}\|\\
&&+\big \| \nabla_yH(x^{k+1},y^{k+1}) - \nabla_yH(x^{k+1},y^{k})\big\|
 +\beta_k\|B_k\| \|y^{k+1}-y^{k}\|\\
 &\le & \mu \nu \|y^{k+1}-y^k\|+L_2(x^{k+1})\|y^{k+1}-y^k\|
 +\beta_k \bar b \|y^{k+1}-y^{k}\|\\
 &=& (\mu \nu+L_2(x^{k+1})+\beta_k \bar b ) \|y^{k+1}-y^{k}\|\\
 &\le & \Big(\mu \nu+\lambda_2^+ +\bar b \gamma_2\rho_2 \lambda_2^+ \Big) \|y^{k+1}-y^{k}\|\\
 &\le &  \Big(\mu \nu+ 2\xi\Big) \vertiii{z^{k+1}-z^{k}}.
 \end{eqnarray*}
Summing up these estimates, we get 
 \[
\vertiii{(d_x^{k+1}, d_y^{k+1})}\le \|d_x^{k+1}\|+\|d_y^{k+1}\|\le (2M+\mu \nu+3\xi)\vertiii{z^{k+1}-z^{k}}.
 \]
 This completes the proof. $\square$\\
\end{proof}

Let $\{z^k=(x^k,y^k)\}_{k\in\mathbb N}$ be a sequence generated by the CPALM
algorithm, starting from the initial point $z^0=(x^o,y^0)$. The set of all limit points is denoted by $\omega(z^0)$, that is,
\[
\omega(z^0)=
\Big\{\bar z\in\mathbb R^n\times \mathbb R^m
:\exists \;\textrm{an increasing sequence of integers} \{k_l\}_{l\in\mathbb N}\;
\textrm{s.t.}\; \lim_{l\to\infty}z^{k_l}= \bar z\Big\}.
\]

\begin{lemma}\label{lma-cluster2}
{\bf (Properties of limit point set $\omega(z^0)$)}
Suppose that Assumptions \ref{ass1}, \ref{ass2}, and \ref{ass3} hold. 
Let $\{z^k\}_{k\in\mathbb N}$ be a sequence generated by the CPALM method 
which is assumed to be bounded. 
Then the following statements are true
\begin{itemize}
\item [(i)] $\omega(z^0) \subset {\rm crit} \;F$
\item [(ii)] $\displaystyle{\lim_{k\to\infty} {\rm dist}} \big(z^k, \omega(z^0)\big)=0$.
\item [(iii)] $\omega(z^0)$ is nonempty, connected, and compact.
\item [(iv)] The function $F$ is finite and constant on $\omega(z^0)$. 
\end{itemize}
\end{lemma}
\begin{proof}
(i) Let $z^*=(x^*,y^*)$ be a limit point of  $\{z^k=(x^k,y^k)\}_{k\in\mathbb N}$.
This means that there is a subsequence $\{(x^{k_l},y^{k_l})\}_{l\in\mathbb N}$
such that $\lim_{l\to\infty }(x^{k_l},y^{k_l})\rightarrow (x^*,y^*)$.
Since $f$ and $g$ are lower semicontinuous, we obtain 
\begin{eqnarray}\label{liminf}
\liminf_{l\to\infty}\; f(x^{k_l})\ge f(x^*)
\quad\quad
\liminf_{l\to\infty}\; g(y^{k_l})\ge g(x^*)
\end{eqnarray}
By (\ref{x-CPALM}), the first step of the CPALM method, for all $k\in\mathbb N$ and $x\in\mathbb R^n$
we have 
\begin{eqnarray*}
&f(x^{k+1})+\langle \nabla_x H(x^k,y^k), x^{k+1}-x^k\rangle+\frac{\alpha_k}{2}\|x^{k+1}-x^k\|_{A_k}^2&
\\[.1in]
&\le f(x)+\langle \nabla_x H(x^k,y^k), x-x^k\rangle+\frac{\alpha_k}{2}\|x-x^k\|_{A_k}^2.&
\end{eqnarray*}
Thus letting $x=x^*$ in the above, we get 
\begin{eqnarray*}
&f(x^{k+1})+\langle \nabla_x H(x^k,y^k), x^{k+1}-x^k\rangle+\frac{\alpha_k}{2}\|x^{k+1}-x^k\|_{A_k}^2&
\\[.1in]
&\le f(x^*)+\langle \nabla_x H(x^k,y^k), x^*-x^k\rangle+\frac{\alpha_k}{2}\|x^*-x^k\|_{A_k}^2.&
\end{eqnarray*}
We choose $k=k_l-1$ in the above inequality and letting $l$ goes to infinity, we obtain
\begin{eqnarray*}
\limsup_{l\to\infty}f(x^{k_l})
\le  f(x^*)+ \limsup_{l\to\infty}\Big(\langle \nabla_x H(x^{k_l},y^{k_l}), x^*-x^{k_l}\rangle+\frac{\alpha_k}{2}\|x^*-x^{k_l}\|_{A_k}^2\Big),
\end{eqnarray*}
where we have used the facts that 
the sequences $\{x^k\}_{k\in\mathbb N}$ and $\{\alpha_k\}_{k\in\mathbb N}$ are bounded,
$\nabla H$ continuous and that the distance between two successive iterates tends to 
zero (see Lemma \ref{CnvProp}(ii)). We also have $ \lim_{l\to\infty} x^{k_l}=x^*$, hence the latter 
inequality reduces to 
$
\limsup_{l\to\infty}f(x^{k_l})\le f(x^*).
$
By this, together with (\ref{liminf}) we obtain 
$\lim_{l\to\infty}f(x^{k_l})= f(x^*).$

Now by the alternative version of iterative step (\ref{y-CPALM}) given in (\ref{VM-PALM-majorant}),
and using (\ref{gstructure}) we have
\begin{eqnarray*}
\begin{array}{l}
g(y^{k+1})+\langle \nabla_y H(x^k,y^k), y^{k+1}-y^k\rangle +
\frac{\beta_k}{2}\|y^{k+1}-y^k\|_{B_k}^2\\[.1in]
\quad\quad\quad\quad\le q(y^{k+1},y^k)+\langle \nabla_y H(x^k,y^k), y^{k+1}-y^k\rangle +
\frac{\beta_k}{2}\|y^{k+1}-y^k\|_{B_k}^2\\[.1in]
\quad\quad\quad\quad\le q(y,y^k)+\langle \nabla_y H(x^k,y^k), y-y^k\rangle +
\frac{\beta_k}{2}\|y-y^k\|_{B_k}^2.
\end{array}
\end{eqnarray*}
We let $y=y^*$ on the right hand side to obtain 
\begin{eqnarray*}
\begin{array}{l}
g(y^{k+1})+\langle \nabla_y H(x^k,y^k), y^{k+1}-y^k\rangle +
\frac{\beta_k}{2}\|y^{k+1}-y^k\|_{B_k}^2\\[.1in]
\quad\quad\le q(y^*,y^k)+\langle \nabla_y H(x^k,y^k), y^*-y^k\rangle +
\frac{\beta_k}{2}\|y^*-y^k\|_{B_k}^2.
\end{array}
\end{eqnarray*}
We choose $k=k_l-1$ in the above inequality and letting $l$ goes to infinity, 
and by the fact that the sequences $\{y^k\}_{k\in\mathbb N}$ and $\{\beta_k\}_{k\in\mathbb N}$
are bounded, $\nabla H$ continuous and that the distance between two successive iterates tends to 
zero (see Lemma \ref{CnvProp}(ii)) we obtain
\begin{eqnarray*}
\begin{array}{lll}
\limsup_{l\to\infty}g(y^{k_l})
&\le& 
\limsup_{l\to\infty}\Big(q(y^*,y^{k_l})+\langle \nabla_y H(x^{k_l},y^{k_l}), y^*-y^{k_l}\rangle +
\frac{\beta_k}{2}\|y^*-y^{k_l}\|_{B_k}^2\Big)\\
&=& 
\limsup_{l\to\infty}\Big((\phi\circ\psi)(y^{k_l})+(\phi'\circ\psi)(y^{k_l})\big(\psi(y^*)-\psi(y^{k_l})\big)
\\
&&\quad\quad\quad\quad\quad\quad
+
\langle \nabla_y H(x^{k_l},y^{k_l}), y^*-y^{k_l}\rangle +
\frac{\beta_k}{2}\|y^*-y^{k_l}\|_{B_k}^2\Big)
\end{array}
\end{eqnarray*}
By the continuity of $\phi\circ\psi$ and $\phi'\circ\psi$
(see Remark \ref{rem3}), Assumption \ref{ass2}(ii), and $ \lim_{l\to\infty} y^{k_l}=y^*$ we obtain 
$\limsup_{l\to\infty}g(y^{k_l})\le 
(\phi\circ\psi)(y^*)=g(y^*)$.
This and (\ref{liminf}) then reduces to 
$\limsup_{l\to\infty}g(y^{k_l})=g(y^*)$.
Therefore, 
\begin{eqnarray*}
\lim_{l\to\infty} F(x^{k_l},y^{k_l})=\lim_{l\to\infty} \Big(f(x^{k_l})+g(y^{k_l})+ H(x^{k_l},y^{k_l})\Big)
=f(x^*)+g(y^*)+H(x^*,y^*)=F(x^*,y^*).
\end{eqnarray*}

On the other hand we know from Lemma \ref{CnvProp}(ii) and  \ref{subgbound}
that $(d_x^k,d_y^k)\in\partial F(x^k,y^k)$
and $(d_x^k,d_y^k)\rightarrow (0,0)$ as $k\to\infty$.
The closedness property of $\partial F$ (see Remark \ref{rem1}(iii)) implies that $(0,0)\in\partial F(x^*,y^*)$.
This proves that $(x^*,y^*)\in {\rm crit} F$.

The proof of (ii), (iii), and (iv) are generic, and are the subsequence of Lemma \ref{CnvProp}(ii). $\square$
\end{proof}

Our objective is now to prove that the sequence generated by the CPALM method 
converges to a critical point of problem (\ref{originalprob})-(\ref{gstructure}). 
For this purpose we consider that the objective function is a KL function (see Section \ref{sub:KL}).

\subsection{Convergence of CPALM to Critical Points}
\begin{theorem}{\bf (Convergence)}\label{conv}
Suppose that $F$ is a KL function such that Assumption \ref{ass1}, \ref{ass2}, and \ref{ass3} hold. 
Let $\{z^k\}_{k\in\mathbb N}$ be a sequence generated by CPALM which is assumed
to be bounded. The following assertions hold.
\begin{itemize}
\item [(i)] The sequence $\{z^k\}_{k\in\mathbb N}$ has finite length, that is,
\begin{eqnarray}\label{finite}
\sum_{k=1}^{\infty}\vertiii{z^{k+1}-z^{k}}<\infty
\end{eqnarray}
\item [(ii)] The sequence $\{z^k\}_{k\in\mathbb N}$ converges to a critical point $z^*=(x^*,y^*)$ of $F$.
\end{itemize}
\end{theorem}

\begin{proof}
Since the sequence $\{z^k\}_{k\in\mathbb N}$ is bounded,
 there exists a converging subsequence $\{z^{k_l}\}_{l\in\mathbb N}$.
 We assume $z^{k_l}\rightarrow z^*$ as $l\to\infty$.
Thus $z^*\in \omega(z^0)\neq\emptyset$.
By Lemma \ref{lma-cluster2}(i), we get that 
\begin{eqnarray}\label{lim}
\lim_{k\to\infty} F(x^k,y^k)=F(x^*,y^*)
\end{eqnarray}
If there exists an integer $\bar k$ for which $F(z^{\bar k})=F(z^*)$
then the inequality (\ref{nondecF}) implies that $z^{k}=z^*$ for all 
$k\ge \bar k$. Thus the sequence  $\{z^k\}_{k\in\mathbb N}$ is an stationary 
sequence and (\ref{finite}) follows.

Since $\{F(z^k)\}_{k\in\mathbb N}$ is a non-increasing sequence,
by (\ref{lim}) we have $F(z^*)<F(z^k)$ for all $k\in\mathbb N$.
Again from (\ref{lim}) for any $\eta>0$
there exists a $k_0$ such that 
$F(z^k)\le F(z^*)+\eta$ for all $k>k_0$.
By Lemma \ref{lma-cluster2}(ii), we know that 
$\lim_{k\to\infty} {\rm dis}(z^k, \omega(z^0))=0$.
This means that for any $\epsilon>0$ there exists a positive integer $k_1$
such that ${\rm dis}(z^k, \omega(z^0))<\epsilon$
for all $k>k_1$. Summing up all these facts, 
we get that $z^k$ belongs to the set 
\begin{eqnarray*}
\Gamma_{\eta,\omega(z^0)}
=\{z: {\rm dist} (z^k, \omega(z^0))<\epsilon
\;\;\;
{\rm and}
\;\;\;
F(z^*)\le F(z^k)\le F(z^*)+\eta
\}
\end{eqnarray*}
for all $k>\hat k=\max\{k_0,k_1\}$.

(i)
Since  $\omega(z^0)$ is nonempty and compact (see Lemma \ref{lma-cluster2}(iii)),
and since $F$ is finite and constant on $\omega(z^0)$  (see Lemma \ref{lma-cluster2}(iv)),
we can apply Lemma \ref{UniKL} with $\Omega=\omega(z^0)$.
Thus, for any $k>\hat k$ we have
\begin{eqnarray*}
\psi'\Big(F(z^k)-F(z^*)\Big)  {\rm dist} \Big(0, \partial F(z^k)\Big)\ge 1.
\end{eqnarray*}
By Lemma \ref{subgbound}, we get that
\begin{eqnarray}\label{xx}
\psi'\Big(F(z^k)-F(z^*)\Big) \ge \frac{1}{2M+\mu\nu+3\xi} |||z^{k}-z^{k-1}|||^{-1}.
\end{eqnarray}
In addition, since $\psi$ is concave, we have,
for every $(u_1,u_2)\in[0,\eta]^2$,
\[
\psi(u_1)-\psi(u_2)\ge \psi'(u_1)(u_1-u_2).
\]
By taking $u_1=F(z^k)-F(z^*)$ and $u_2=F(z^{k+1})-F(z^*)$, we obtain
\begin{eqnarray}\label{ff}
\begin{array}{l}
\psi\Big(F(z^k)-F(z^*)\Big)-\psi\Big(F(z^{k+1})-F(z^*)\Big)\\
\quad\quad\ge\psi'\Big(F(z^k)-F(z^*)\Big)\Big(F(z^k)-F(z^{k+1})\Big) 
\end{array}
\end{eqnarray}
For convenience, we define for all $p,q\in\mathbb N$ and $z^*$ the following quantity
\[
\Delta_{p,q}= \psi\Big(F(z^p)-F(z^*)\Big)-\psi\Big(F(z^q)-F(z^*)\Big)
\]
Then, combining (\ref{xx}) and (\ref{ff}) and Lemma \ref{CnvProp}(i)
for any $k>\hat k$ we have
\begin{eqnarray*}
\Delta_{k,k+1}\ge \frac{\delta}{2M+\mu\nu+3\xi}\frac{|||z^{k+1}-z^k|||^2}{|||z^{k}-z^{k-1}|||}
\end{eqnarray*}
We  rearrange this to obtain 
\begin{eqnarray*}
|||z^{k+1}-z^k|||^2 \le C \Delta_{k,k+1}|||z^{k}-z^{k-1}|||
\end{eqnarray*}
where 
$
C=  (2M+\mu\nu+3\xi)/\delta.
$
Using the fact that $\sqrt{\alpha\beta}\le (\alpha+\beta)/2$ for all $\alpha,\beta\ge 0$ we obtain
\begin{eqnarray}\label{cc}
2|||z^{k+1}-z^k||| \le |||z^{k}-z^{k-1}|||+ C \Delta_{k,k+1}
\end{eqnarray}
Summing up (\ref{cc}) for $k> \hat k$ yields
\begin{eqnarray*}
2\sum_{i=\hat k+1}^k |||z^{i+1}-z^i||| &\le&
\sum_{i=\hat k+1}^k  |||z^{i}-z^{i-1}|||+ C\sum_{i=\hat k+1}^k \Delta_{i,i+1}\\
&\le & \sum_{i=\hat k+1}^k  |||z^{i+1}-z^{i}|||+  |||z^{\hat k+1}-z^{\hat k}|||+C\sum_{i=\hat k+1}^k \Delta_{i,i+1}\\
&=&\sum_{i=\hat k+1}^k  |||z^{i+1}-z^{i}|||+  |||z^{\hat k+1}-z^{\hat k}|||+C\sum_{i=\hat k+1}^k \Delta_{\hat k+1,i+1}
\end{eqnarray*}
where the last inequality follows the fact that $\Delta_{p,q}+\Delta_{q,r}=\Delta_{p,r}$
for all $p,q,r\in\mathbb N$. Since $\psi\ge 0$, we thus have for any $k\ge \hat k$ that
\[
\sum_{i=\hat k+1}^k |||z^{i+1}-z^i|||\le |||z^{\hat k+1}-z^{\hat k}|||+
C\psi\Big(F(z^{\hat k+1})-F(z^*)\Big).
\] 
This shows that the sequence $\{z^k\}_{k\in\mathbb N}$ has finite length, that is 
\begin{eqnarray}\label{flength}
\sum_{i=1}^{\infty} |||z^{k+1}-z^k|||<\infty.
\end{eqnarray}

(ii) The result (\ref{flength}) implies that the sequence $\{z^k\}_{k\in\mathbb N}$
is Cauchy, thus it converges. Then by Lemma \ref{lma-cluster2} (i) 
the limit point is a critical point of $F$.

\end{proof}

\subsection{Extension of CPALM }
We consider an extension of CPALM method to the more general setting of the form 
\begin{eqnarray}\label{cPALMex}
\min_{x \in R^{n}, y_j\in\mathbb R^{m_j}} \Psi(x, y_1,\dots, y_q):=
f(x)+g(y_1,\dots,y_p)+H(x,y_1,\dots,y_p), 
\end{eqnarray} 
where  $H:\mathbb R^n\times \mathbb R^M$ with $M=\sum_{j=1}^q m_j$
is assumed to be a smooth function,  $f:\mathbb R^n\to]-\infty,+\infty]$ 
and $g:\mathbb R^M\to]-\infty,+\infty]$ are proper and lower-semicontinuous functions. 
More specifically, the function $g$ is a sum
of composite functions as follows
\begin{eqnarray*}
\forall y\in\mathbb R^M\quad g(y_1,\dots,y_p)=\sum_{j=1}^q g_j(y_j)=
\sum_{j=1}^q (\phi_j\circ \psi_j)(y_j). 
\end{eqnarray*}
where for every 
$j\in\{1,2,\dots,q\}$,
$\psi_j:\mathbb R^n\to [0,+\infty]$
is convex, proper, lower semi-continuous and Lipschitz continuous on its domain,
and $\phi_j:[0,+\infty]\to]-\infty,+\infty]$ is concave, strictly increasing and differentiable function, such that 
$(\phi_j\circ\psi_j)$ is Lipschitz-continuous on the domain of $\psi_j$.

For simplicity of the presentation of the algorithm for the case of more than two blocks we will
use the following notations.
For $y=[y_1,\dots,y_q]\in\mathbb R^M$, $y_{<j}:=[y_1; \dots; y_{j-1}]\in\mathbb R^{m_1+\dots+m_{j-1}}$
and $y_{>j}:=[x_{j+1};\dots;y_q]\in\mathbb R^{m_{j+1}+\dots+m_q}$
(clearly, $x_{<0}$ and $x_{>p}$ are null variables, which may be used for notational ease.)
By the multi-block variant of (\ref{VM-PALM}) we have 
\begin{eqnarray}\label{VM-PALMex}
\left\{
\begin{array}{l}
x^{k+1}\in\displaystyle{\arg\min_{x\in\mathbb R^n}} \;\;\;
f(x)+\Big\langle \nabla_x H(x^k,y^k), x-x^k\Big\rangle
+\frac{\alpha_k}{2}\Big\|x-x^k\Big\|_{A_k}^2,
\\[.1in]
y_1^{k+1}\in\displaystyle{\arg\min_{y_1\in\mathbb R^{m_1}}} \;\;\;(\phi_1\circ\psi_1)(y_1)
+\Big\langle \nabla_{y_1} H(x^{k+1},y_1^k,y_{j>1}^k), y_1-y_1^k\Big\rangle 
+\frac{\beta_k^1}{2}\Big\|y_1-y_1^k\Big\|^2_{B_k^1} 
\\[.1in]
\vdots\\[.1in]
y_j^{k+1}\in\displaystyle{\arg\min_{y_j\in\mathbb R^{m_j}}} \;\;\;(\phi_j\circ\psi_j)(y_j)
+\Big\langle \nabla_{y_j} H(x^{k+1}, y^{k+1}_{<j}, y_j^k, y^k_{>j}), y_j-y_j^k\Big\rangle 
+\frac{\beta_k^j}{2}\Big\|y_j-y_j^k\Big\|^2_{B_k^j} 
\\[.1in]
\vdots
\\[.1in]
y_p^{k+1}\in\displaystyle{\arg\min_{y_p\in\mathbb R^{m_p}}} \;\;\;(\phi_p\circ\psi_p)(y_p)
+\Big\langle \nabla_{y_p} H(x^{k+1},y_{<p}^{k+1},y_p^k), y_p-y_p^k\Big\rangle 
+\frac{\beta_k^p}{2}\Big\|y_p-y_p^k\Big\|^2_{B_k^p} 
\end{array}
\right.
\end{eqnarray}
where $A_k\in\mathbb R^{n\times n}$ and $B_k^j\in\mathbb R^{m_j\times m_j}$
for $j=1,\dots,p$ are symmetric positive definite matrices,
and $\{\alpha_k\}_{k\in\mathbb N}$
and $\{\beta_k^j\}_{k\in\mathbb N}$ for $j=1,\dots,p$
are positive real sequences.
Due to the composite form of $g_j$ the solution of subproblem might not be 
computable, either efficiently or at all. Thus, to overcome this issue, we 
replace at each iteration $k\in\mathbb N$, the function 
$g_j$ by an approximation denoted by $q_j(y_j,y_j^k):\mathbb R^{m_j}\to ]-\infty,+\infty]$
a majorant function of $(\phi_j\circ\psi_j)$ at $y_j^k$,
\begin{eqnarray*}
(\forall y_j\in\mathbb R^{m_j})
\quad
\left
\{
\begin{array}{l}
(\phi_j\circ\psi_j)(y_j)\le q_j(y_j,y_j^k)\\[.1in]
(\phi_j\circ\psi_j)(y_j^k)=q_j(y_j^k,y_j^k)
\end{array}
\right.
\end{eqnarray*}
and is obtained by taking the tangent of the concave differentiable
function $\phi_j$ at $\psi_j(y_j^k)$ for every $j\in\{1,\dots,p\}$ and $k\in\mathbb N$:
\begin{eqnarray*}
(\forall y_j\in\mathbb R^{m_j})
\quad
q_j(y_j,y_j^k)= 
(\phi_j\circ\psi_j)(y_j^k)
+\Upsilon_k^j\Big(\psi_j(y_j)-\psi_j(y_j^k)\Big).
\end{eqnarray*}
where 
\begin{eqnarray}\label{upsilon}
\Upsilon_k^j:=(\phi'\circ\psi_j)(y_j^k).
\end{eqnarray}
\vspace{.2in}
The multi-block version of CPALM method, called Multi-CPALM, is given as follows 

{\bf Multi-CPALM: Multi-Block CPALM }
\begin{itemize}
\item[1.] Initialization: choose a starting point $(x^0,y_1^0, \cdots, y_p^0)\in\mathbb R^n\times \mathbb R^M$.
\item[2.] For each $k=0,1,\dots$ generate a sequence $\{(x^k,y^k)\}_{k\in\mathbb N}$ as follows
\begin{itemize}
\item[2.1.] Take $\rho_1=\max\{1,\frac{1}{\lambda_{\min}(A_k)}\}$, $\gamma_1>1$, set $\alpha_k=\gamma_1\rho_1L_1(y^k)$ and compute
\begin{eqnarray*}
x^{k+1}\in\displaystyle{\arg\min_x} \;
f(x)+\Big\langle \nabla_x H(x^k,y_1^k,\cdots,y_p^k), x-x^k\Big\rangle
+\frac{\alpha_k}{2}\Big\|x-x^k\Big\|_{A_k}^2,
\end{eqnarray*}
\item[2.2.] 
For $j=1,\dots, p$,
compute $\Upsilon_k^j:=(\phi'\circ\psi_j)(y_j^k)$,
take $\rho_2^j=\max\{1,\frac{1}{\lambda_{\min} (B_k^j)}\}$,
$\gamma_2^j>1$, 
set $\beta_k^j=\gamma_2^j\rho_2^j L_2^j(x^{k+1}, y_{<j}^{k+1},y_{>j}^{k})$
and compute
\begin{eqnarray*}
y_j^{k+1}\in\displaystyle{\arg\min_{y_j}} \;
\Upsilon_k^j\psi_j(y_j)
+\big\langle \nabla_{y_j} H(x^{k+1}, y^{k+1}_{<j}, y_j^k, y^k_{>j}), y_j-y_j^k\big\rangle 
+\frac{\beta_k^j}{2}\Big\|y_j-y_j^k\Big\|^2_{B_k^j}
\end{eqnarray*}

\end{itemize}
\end{itemize}
Theorem \ref{conv} can be simply extended and applied for the Multi-CPALM.

\section{Simulations}\label{simulation}
In this section we consider two different models arising from Parallel MRI application to illustrate the numerical performance of the
proposed CPALM method. All results have been implementedin MTLAB 2020a and executed 
on a Macbook Pro 1.4 GHz Quad-Core Intel Core i5.

\subsection{Parallel MRI}
Parallel MRI is a technique that exploits the differences in the spatial sensitivity of multiple receiver coils
acquired simultaneously to localize signal. 
The subsampling $k$-space reduces the acquisition time significantly in comparison to the traditional MRI,
which leads to decreasing motion related artifacts, reducing breath-hold time, and 
shorter duration of diagnostic exam. 

Assume $u\in\mathbb R^{M\times N}$ denotes the image, and $u_{i,j}$ 
denotes the intensity of the  gray level at the $(i,j)$ pixel. 
In Parallel MRI involving $N_c$ radio frequency coils, the operator 
$A=[A_1, A_2, \dots, A_{N_c} ]$ is defined by 
$
A: u\to [A_1u; A_2u,\dots; A_{N_c} u],
$
where $A_i:R^{M\times N}\to R^{M\times N}$ for $i=1,\dots, N_c$ given by
\[
A_i u=P\circ \mathcal F \big(S_i\circ u\big),
\]
where $P$, $\mathcal F$, and $S_i$ are operators on $\Omega$.
The operator $P$ is the under-sampling pattern, called trajectory or mask,
$\mathcal F$ denotes the Fourier transform operator, and 
$S_i$ is the sensitivity map of the $i$-th coil.
The notation $\circ$ is the Hadamard product between two matrices
and $[\cdot; \cdot]$ means stacking the operators above each other. 
The observed data $\hat u=[\hat u_1; \hat u_2; \dots, \hat u_{N_c}]$
with $\hat u_i$ is corresponding under-sampled image from the $i$th coil 
defined by
\[
\hat u_i=\mathcal A_iu+\xi_i
\]
where $u$ is the true image and $\xi_i$ is 
the observation error. Note that the true image $u$ and artifacts $\xi_i$, 
for $i=1,\dots,N_c$ are unknown. 
\begin{figure}
 \begin{center}
\begin{tabular}{ccc} 
{\includegraphics[width=.2\textwidth]{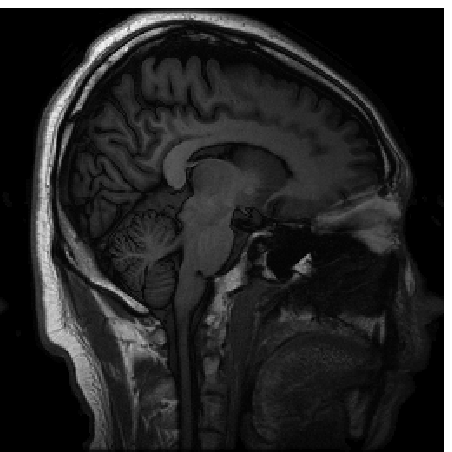}} &
{\includegraphics[width=.2\textwidth]{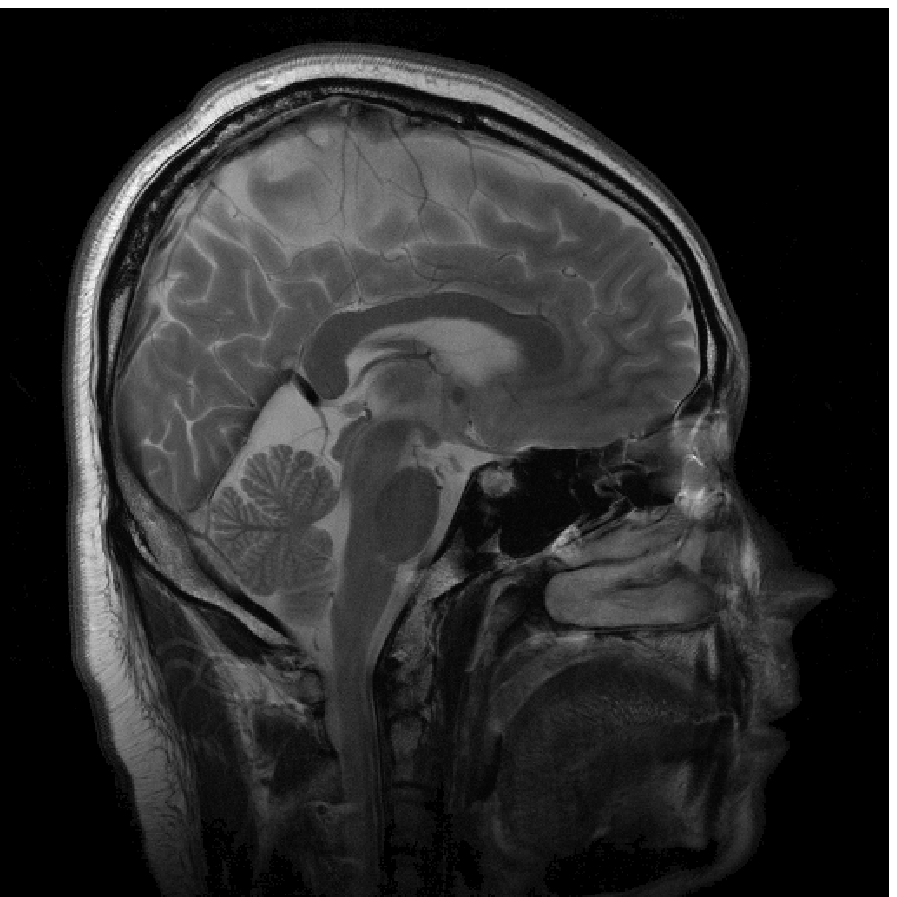}}&
{\includegraphics[width=.2\textwidth]{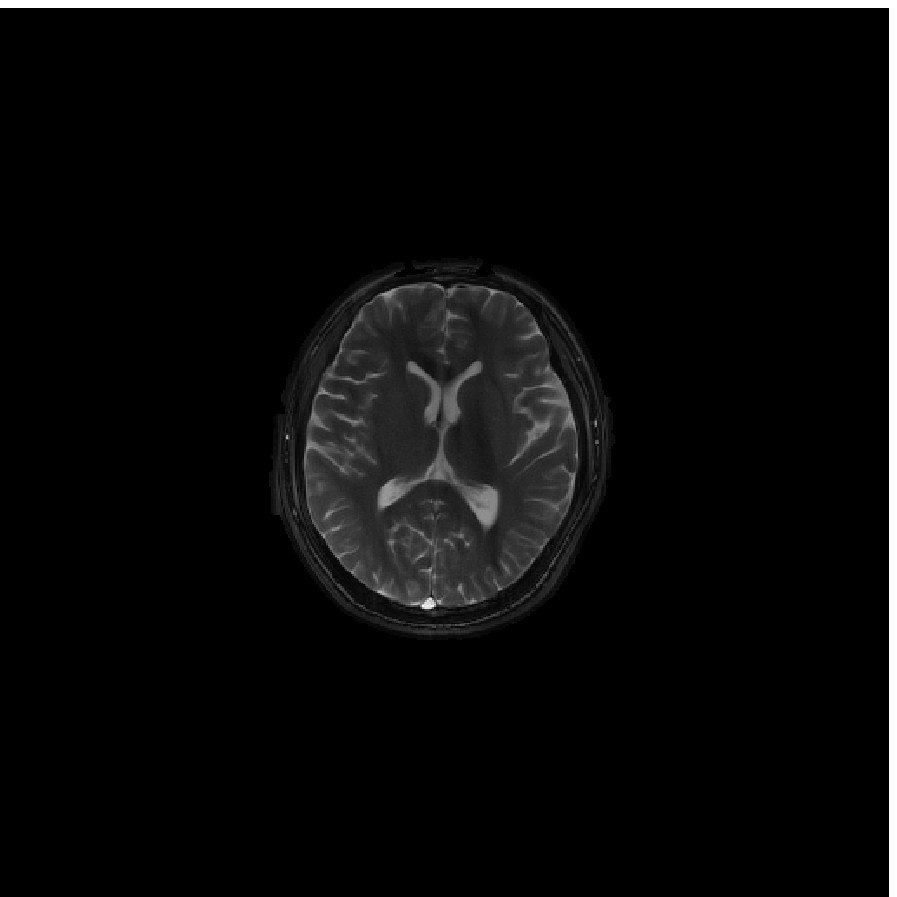}} \\
(a) data1 & (b) data2 & (c) data3 \\
\end{tabular}
\end{center}
\caption {Ground truth images.}
\label{Fig:data} 
\end{figure}

For numerical experiments, we consider the data1, data2, and data3
shown in Figure \ref{Fig:data} and 
their acquisition parameters given in Table \ref{acq}.
 \begin{table}
 {\small
\begin{center}
 \begin{tabular}{ccccccc} 
\toprule
Ground truth &  size ($\times 8$)& FOV (mm$^2$) & TR (mm)  & TE
(ms) & slice thickness  &  flip angle\\\midrule
(a) data1 & $256\times 256$ &220 & 3060 ms &126 ms & 5& 90$^\circ$  \\
(b) data2 & $512 \times 512$ & 205& 3000 ms &85 ms & 5 & 90$^\circ$  \\
(c) data3 & $512 \times 512$ &220& 53.5 ms &3.4 ms & 5& 75$^\circ$  \\
 \bottomrule
  \end{tabular}
   \end{center}
    \caption{The acquisition parameters for data1, data2, and data3.}\label{acq}}
\end{table}
For all three data sets, the ground truth image from a $N_c=8$ channel coil is given by
\[
{u}_{i,j}^*=\Big({\sum_{k=1}^{N_c}\|{u}_{ij}(k)\|^2}\Big)^{1/2},
\]
where ${u}_{i,j}(k)$ is the $i,j$-th component of the Fourier transform associated
with the full $k$-space data on the $k$-th channel.
In acquiring data1 and data2, a Poisson random mask $P$ with
a 25\% undersampling, and for data3  a radial mask with a 
34\% undersampling ratio is used. Figure \ref{Fig:mask} displays these 
undersampling patterns.
\begin{figure}[h]
 \begin{center}
 \begin{tabular}{cc}
{\includegraphics[width=.2\textwidth]{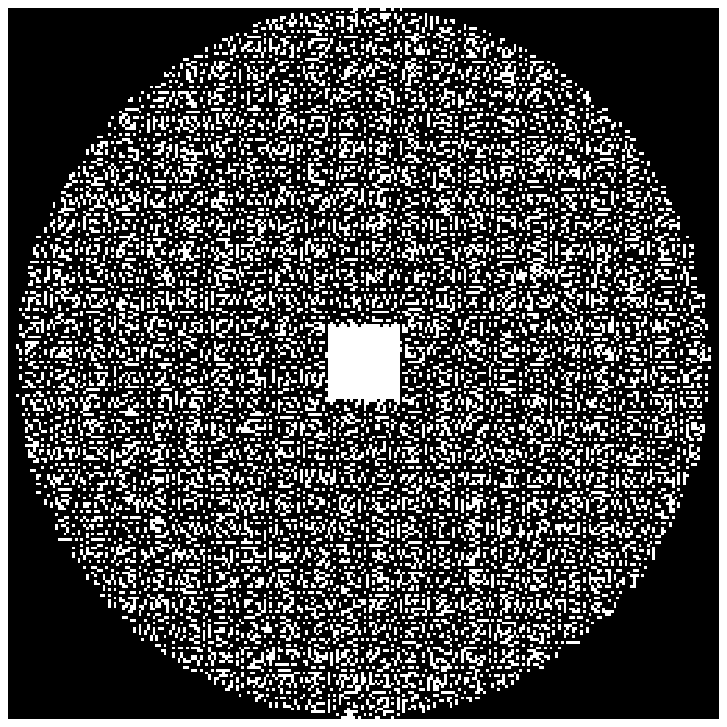}}
&{\includegraphics[width=.2\textwidth]{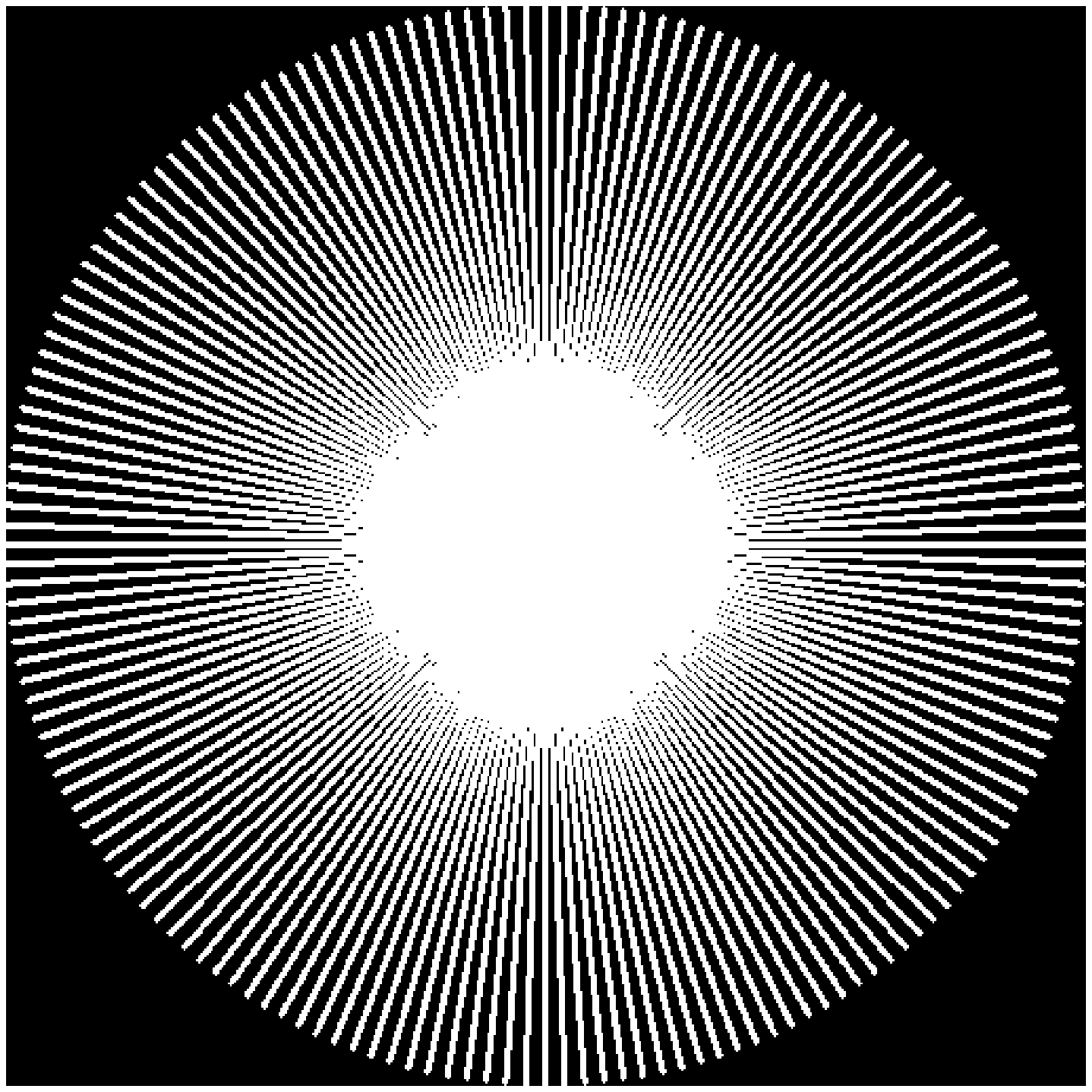}}\\
(a)&(b)
\end{tabular}
 \end{center}
 \caption{
(a) Poisson mask, with 25\% undersampling ratio, (b) Radial mask, with 34\% undersampling ratio.}
\label{Fig:mask}
\end{figure}
Due to the undersampling affect in MRI, the integrated images from all coils
has noise and artifact, as we see in Figure \ref{Fig:datanoisy} and they need to be removed
through the reconstruction process to obtain high-quality images.
\begin{figure}
  \centering
   \begin{center}
\begin{tabular}{ccc} 
{\includegraphics[width=.2\textwidth]{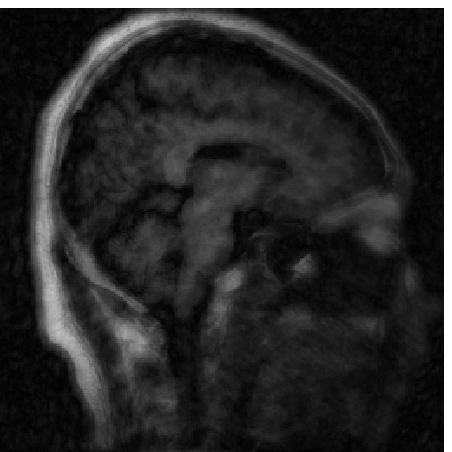}}&
{\includegraphics[width=.2\textwidth]{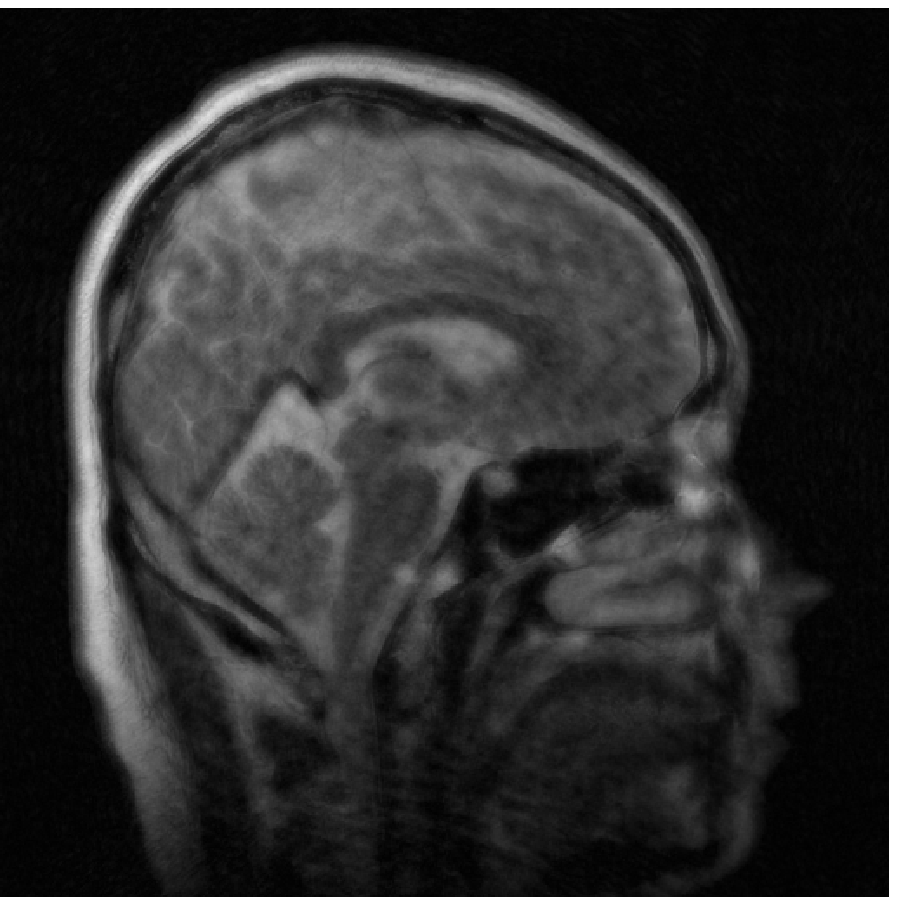}}&
{\includegraphics[width=.2\textwidth]{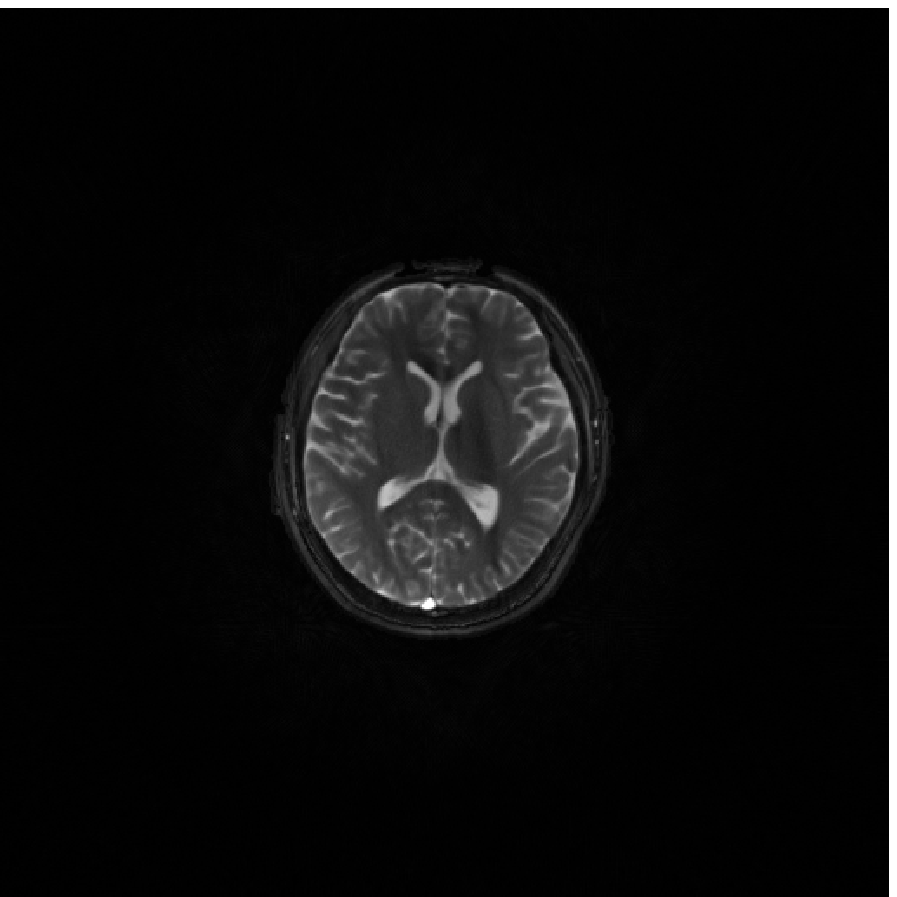}} \\
(a)  & (b)  & (c)\\
\end{tabular}
\end{center}
  \begin{minipage}[b]{1.0\linewidth}
    \begin{center}
    \begin{tabular}{cccc} 
\toprule
Observed images ($\hat u$):  &  (a) data1 & (b) data2 & (c) data3 \\\midrule
SNR &10.14 &15.72&17.92 \\
PSNR &23.17 &27.19&40.04\\
 \bottomrule
    \end{tabular}
    \end{center}
      \end{minipage}
\caption {Observed integrated images obtained from MRI machine, showing noise and artifacts.}
\label{Fig:datanoisy} 
\end{figure}

We measure the quality of images with SNR (Signal to Noise Ratio) and PSNR (Peak Signal to Noise Ratio) 
defined by
\[
{\rm SNR}:= 10 \log_{10}\frac{\|u\|^2_{F}}{\|u-u_0\|^2_F}
\quad 
{\rm and}\quad
{\rm PSNR}:= 10 \log_{10}\frac{\|u\|^2_{F}}{\|u-u_0\|^2_F/n},
\] 
where $u_0$ is the ground truth image, without noise, $n=M\times N$ is the total numbers of pixels in the image,
and $\|.\|_F$ is the Frobenius norm.  The higher value of SNR and PSNR stands for better image quality.

The main objective is to
minimize the discrepancy between the observed perturbed output
$\hat u$ from $A u$, the predicted output of the linear  model 
$u\mapsto Au$.~Most notably, minimizing the quadratic penalty function
 \begin{eqnarray}\label{f}
\min_{u}\frac 12  \| Au-\hat u\|_F^2 
 \end{eqnarray}
where $\|\cdot\|$ is the $\ell_2$ norm.
Some regularizations is often added to the data fidelity term to make the problem well-posed
 \cite{chy13,hnyz15,hyz13,yashtini19,yhcy12}. In the following, we consider two regularization types 
 thus two different models and describe the Multi-CPALM method to solve them. 

\subsection{Log-sum penalization}\label{logsum}
The first model that we consider is the log-sum model given below
\begin{eqnarray}\label{logsum-prob}
\min_{u\in\mathbb R^{M\times N}}
\frac{\lambda}{2} \|Au-\hat u\|_F^2 +\frac{1}{2\mu} \sum_{i,j} \log\big(1+\mu \|(Du)_{i,j}\|^2\big)
\end{eqnarray} 
where $\hat u\in\mathbb R^{M\times N\times k}$  is the observed MRI images, with noise and artifacts.
Here $M\times N$ is the size of image and $k$ is the number of coils used in MRI machine.  
$\|Du\|$ is the matrix of coordinates $\|(Du)_{i,j}\|:=\sqrt{((D_xu)_{i,j})^2+((D_yu)_{i,j})^2}$,
where $D_xu$ is a discrete implementation of the $x$-derivative
of the image
and $D_yu$ is a discrete implementation of the $y$-derivative
of the image, (both considered as a function $\mathbb R^2 \to \mathbb R$).

We introduce a new variable $w=(w^1,w^2)\in\mathbb R^{M\times N\times 2}$, 
$w_{i,j}=(Du)_{i,j}=\big((D_xu)_{i,j}, (D_yu)_{i,j}\big)$ and a proper penalization 
to  obtain 
\begin{eqnarray}\label{sim-prob}
\min_{u}
\frac{\lambda}{2} \|Au-\hat u\|_F^2 +
\frac{1}{2\mu}\sum_{i,j} \log\big(1+\mu\big \|w_{i,j}\big \|^2\big)
+\frac{\tau}{2} \sum_{i,j} \big\|w_{i,j}-(Du)_{i,j}\big\|^2
\end{eqnarray} 
where $\tau>0$. 
Comparing (\ref{sim-prob}) with (\ref{cPALMex}), we have $x=u$ and instead of $y_j$ we have $w_{i,j}$, 
and
\begin{eqnarray*}
&f(u)= \frac{\lambda}{2} \|Au-\hat u\|_F^2, \quad
g(w)=\frac{1}{2\mu}\sum_{i,j} \log\big(1+\mu\big \|w_{i,j}\big \|^2\big)&
\\[.1in]
&H(u,w)=\frac{\tau}{2} \sum_{i,j} \big\|w_{i,j}-(Du)_{i,j}\big\|^2,
\quad
g(w)=\sum_{i,j} g_{i,j} (w_{i,j})
=\sum_{i,j} (\phi_{i,j}\circ\psi_{i,j}) (w_{i,j}) 
&
\\[.1in]
&\phi_{i,j}(x)=\frac{1}{2\mu}\log(1+\mu x^2),\quad
\psi_{i,j}(w_{i,j})=\|w_{i,j}\|. &
\end{eqnarray*}
We execute the Multi-CPALM with the following subproblems to solve (\ref{sim-prob}):
\begin{eqnarray}\label{CPALM:ex1-u}
u^{k+1}= \displaystyle{\arg\min_{u}}
\Big\{\frac{\lambda}{2} \|Au-\hat u\|_F^2+\tau\Big \langle D\tr (Du^{k}-w^k), u\Big\rangle +
\frac{\alpha_k}{2} \Big\|u-u^k\Big\|^2_{A_k} \Big\}
\end{eqnarray}
and for all $i,j$ we update 
\begin{eqnarray}\label{CPALM:ex1-w}
w^{k+1}_{i,j}= \displaystyle{\arg\min_{w_{i,j}}}
\Big\{\Upsilon^{i,j}_{k} \Big\|w_{i,j}\Big\| + 
\tau \Big\langle w_{i,j}^{k}-(Du^{k+1})_{i,j}, w_{i,j}\Big\rangle
+\frac{\beta^{i,j}_k}{2} \Big\|w_{i,j}-w_{i,j}^k\Big\|^2_{B_k^{i,j}}\Big\},
\end{eqnarray}
where $\Upsilon^{i,j}_{k}=\|w_{i,j}^k\|/(1+\mu \|w_{i,j}^k\|^2)$. 
We choose $B_k^{i,j}$ to be identity matrices for all $k\in\mathbb N$,
and all $i=1,\dots,M$ and $j=1,\dots,N$.
By the fact that the matrix $A$ and as the result $A\tr A$ are  ill-conditioned,
it is not practical to invert them. 
To eliminate this issue, we consider the symmetric
variable matrices $A_k= \frac{\delta_k}{\alpha_k} I -\frac{\lambda}{\alpha_k} A\tr A$, with 
$\delta_k > \lambda \rho(A\tr A)$ where $\rho(A\tr A)$ denotes
the spectral radius of $A\tr A$. With this consideration, the optimality condition of 
$u$-subproblem (\ref{CPALM:ex1-u}) leads us to a closed form solution for $u^{k+1}$ as follows
\begin{eqnarray}\label{numex1:u}
u^{k+1}= u^k-\delta^{-1} \Big(\lambda A \tr (A u^k-\hat u) + \tau D\tr (Du^k-w^k)\Big).
\end{eqnarray}
The $w_{i,j}$ subproblem (\ref{CPALM:ex1-w}) is non-smooth and can be rewritten as follows 
\begin{eqnarray*}
w^{k+1}_{i,j} = \displaystyle{\arg\min_{w_{i,j}}}
\Big\{
\Upsilon^{i,j}_k \|w_j\| + 
\frac{\beta^{i,j}_{k}}{2} 
\Big\| w_{i,j}-w_{i,j}^k +\frac{\tau}{\beta^{i,j}_k} \big(w_{i,j}^{k}-(Du^{k+1})_{i,j}\big)\Big\|^2
\Big\}
\end{eqnarray*}
and solved by the shrinkage formula in a closed form 
\begin{eqnarray}\label{numex1:w}
w^{k+1}_{i,j}={\rm shrink} \Big\{ w_{i,j}^k -\frac{\tau}{\beta^{i,j}_k}
 \big(w_{i,j}^{k}-(Du^{k+1})_{i,j}\big), 1/ \Upsilon^{i,j}_k \Big\}, \forall i,j
\end{eqnarray}
where ${\rm shrink} (t,\mu)=\frac{t}{\|t\|} \max \{\|t\|-\frac{1}{\mu}, 0\}$ with the convention $(\frac{0}{|0|}=0)$.

\subsection{$\ell_p^p$ penalization}\label{lp}
In this section we show that the CPALM method can be used to solve nonconvex $\ell_p^p$ norms,
where $p\in ]0,1[$, given by
\begin{eqnarray*}
\|x\|_p=\Big(\sum_{i=1}^N |x_i|^p\Big)^{1/p}\quad  \;\;\forall x\in\mathbb R^N.
\end{eqnarray*}
The $\ell_p^p$ penalization model for MRI reconstruction is given by 
\begin{eqnarray}\label{lp-prob}
\min_{u\in\mathbb R^{M\times N}}\;\; \frac{\lambda}{2} \|Au-\hat u\|_F^2+ \theta \sum_{i,j} \big\|(Du)_{i,j}\big\|^p. 
\end{eqnarray}
where $\theta>0$.
By introducing a new variable $w=(w^1,w^2)\in\mathbb R^{M\times N\times 2}$, 
$w_{i,j}=(Du)_{i,j}=\big((D_xu)_{i,j}, (D_yu)_{i,j}\big)$ and a proper penalization 
we obtain 
\begin{eqnarray}\label{eq-ex3}
\min_{u}\;\; \frac{\lambda}{2} \|Au-\hat u\|_F^2+
\theta \sum_{i,j} \big\|w_{i,j}\big\|^p
+\frac{\tau}{2}\sum_{i,j}\big \|w_{i,j}-(Du)_{i,j}\big\|^2
\end{eqnarray}
where $\tau>0$. 
Comparing (\ref{lp-prob}) with (\ref{cPALMex}), we have $x=u$ and instead of $y_j$ we have $w_{i,j}$, 
and
\begin{eqnarray*}
&f(u)= \frac{\lambda}{2} \|Au-\hat u\|_F^2, \quad
g(w)=\theta \sum_{i,j} \big\|w_{i,j}\big\|^p
&
\\[.1in]
&H(u,w)=\frac{\tau}{2} \sum_{i,j} \big\|w_{i,j}-(Du)_{i,j}\big\|^2,
\quad
g(w)=\sum_{i,j} g_{i,j} (w_{i,j})
=\sum_{i,j} (\phi_{i,j}\circ\psi_{i,j}) (w_{i,j}) 
&
\\[.1in]
&\phi_{i,j}(x)=\theta x^p,\quad
\psi_{i,j}(w_{i,j})=\|w_{i,j}\|. &
\end{eqnarray*}
We exploit the Multi-CPALM with the following subproblem to solve (\ref{lp-prob})
\begin{eqnarray}\label{CPALM:ex2}
\begin{array}{lll}
u^{k+1}&=& \displaystyle{\arg\min_{u}}
\Big\{\frac{\lambda}{2} \|Au-\hat u\|_F^2+\tau \Big\langle D\tr (Du^{k}-w^k), u\Big\rangle +
\frac{\alpha_k}{2} \Big\|u-u^k\Big\|^2_{A_k} \Big\}
\\[.1in]
w^{k+1}_{i,j}&=& \displaystyle{\arg\min_{w_{i,j}}}
\Big\{\Upsilon^{i,j}_k \Big\|w_{i,j}\Big\| + 
\tau \Big\langle w_{i,j}^{k}-(Du^{k+1})_{i,j}, w_{i,j}\Big\rangle
+\frac{\beta^{i,j}_k}{2} \Big\|w_{i,j}-w_{i,j}^k\Big\|_{B^{i,j}_k}^2\Big\},
\;\;\forall i,j
\end{array}
\end{eqnarray}
where $\Upsilon^{i,j}_k= \theta p \|w_{i,j}\|^{p-1}$.
Similar to previous example, we consider $A_k= \frac{\delta_k}{\alpha_k} I -\frac{\lambda}{\alpha_k} A\tr A$, with $\delta_k > \lambda \rho(A\tr A)$, and $B_k^{i,j}$ as identity matrices for all $i,j$ and $k$. 
Thus the $u$-subproblem is solved exactly by (\ref{numex1:u}) and the $w$-subproblem by (\ref{numex1:w}), but
$\Upsilon^{i,j}_k= \theta p \|w_{i,j}\|^{p-1}$.

\subsection{Numerical Results and Comparison}
In this section, we compare the reconstruction results by the CPALM method for solving 
two different models (\ref{logsum-prob}) and (\ref{lp-prob}), discussed in 
Sections \ref{logsum} and \ref{lp}.  
We consider the following parameter values:
 $\lambda=1000$, $\mu=0.0001$, $\theta=0.0001$, $\tau=1$, $\delta=1000$, $\beta=10$, and $p=0.5$.

Figure \ref{Fig:logsum} and \ref{Fig:lp} show
the results obtained by the CPALM method for 
solving  the log-sum penalization model (\ref{logsum-prob})
and the $\ell_p$ penalization model (\ref{lp-prob}), respectively. 
(a), (b), and (c) demonstrates the reconstruction results. 
Comparing SNR and PSNR with Table \ref{Fig:datanoisy} we observe
a significant improvement obtained by the proposed method.
The image relative error defined by $\|u^k-u_0\|/\sqrt{n}\|u_0\|$
versus CPU (sec.) time are also shown  in 
parts (d), (e), (f). We observe that all trajectories
are monotonically decreasing. 
\begin{figure}
 \begin{center}
\begin{tabular}{ccc} 
{\includegraphics[width=.2\textwidth]{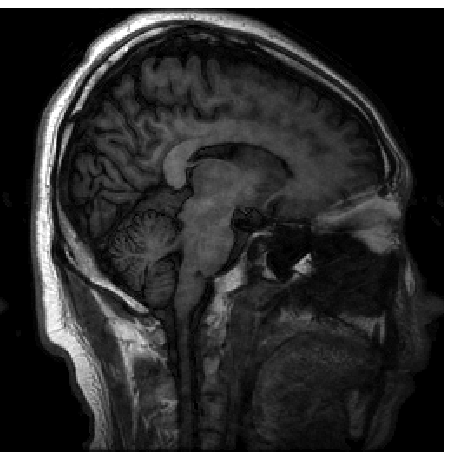}}&
{\includegraphics[width=.2\textwidth]{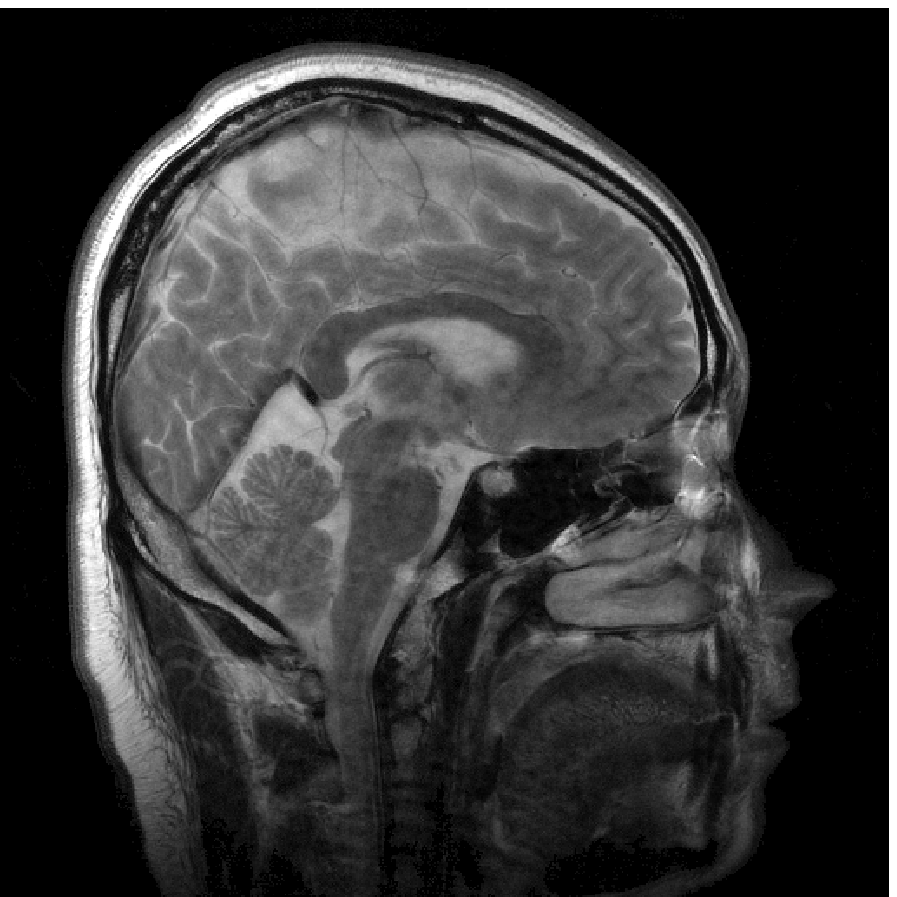}}&
{\includegraphics[width=.2\textwidth]{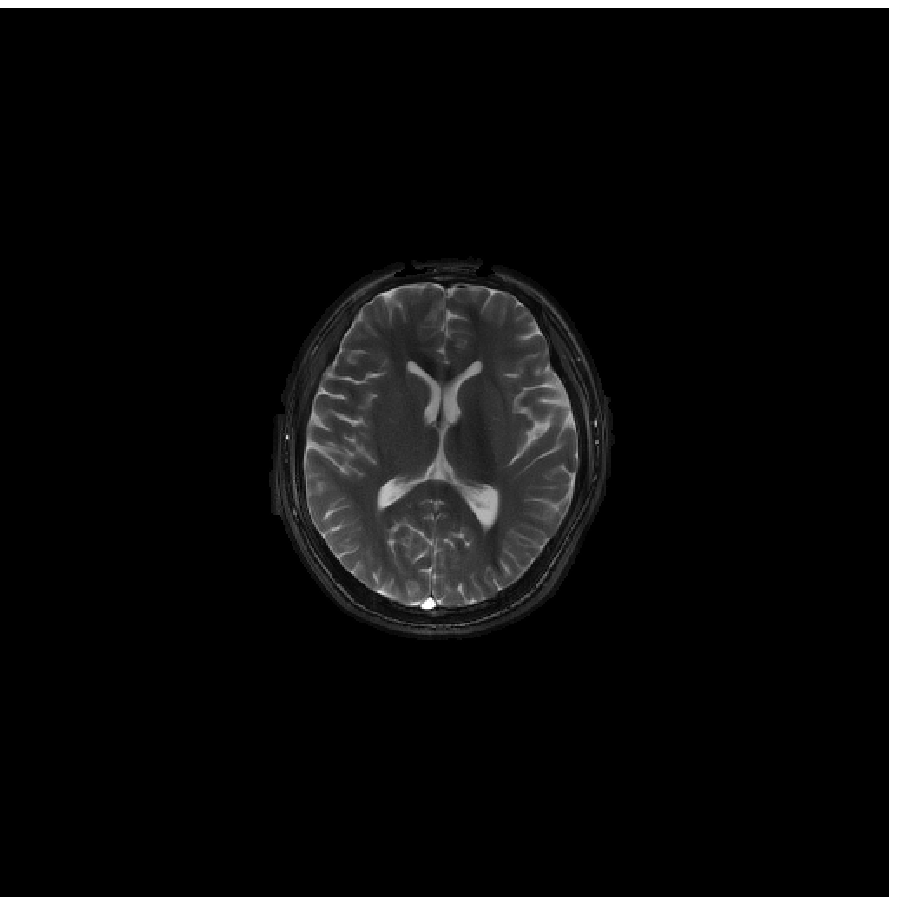}}
 \\
(a) Restored data1 & (b) Restored data2& (c) Restored data3\\
{\includegraphics[width=.25\textwidth]{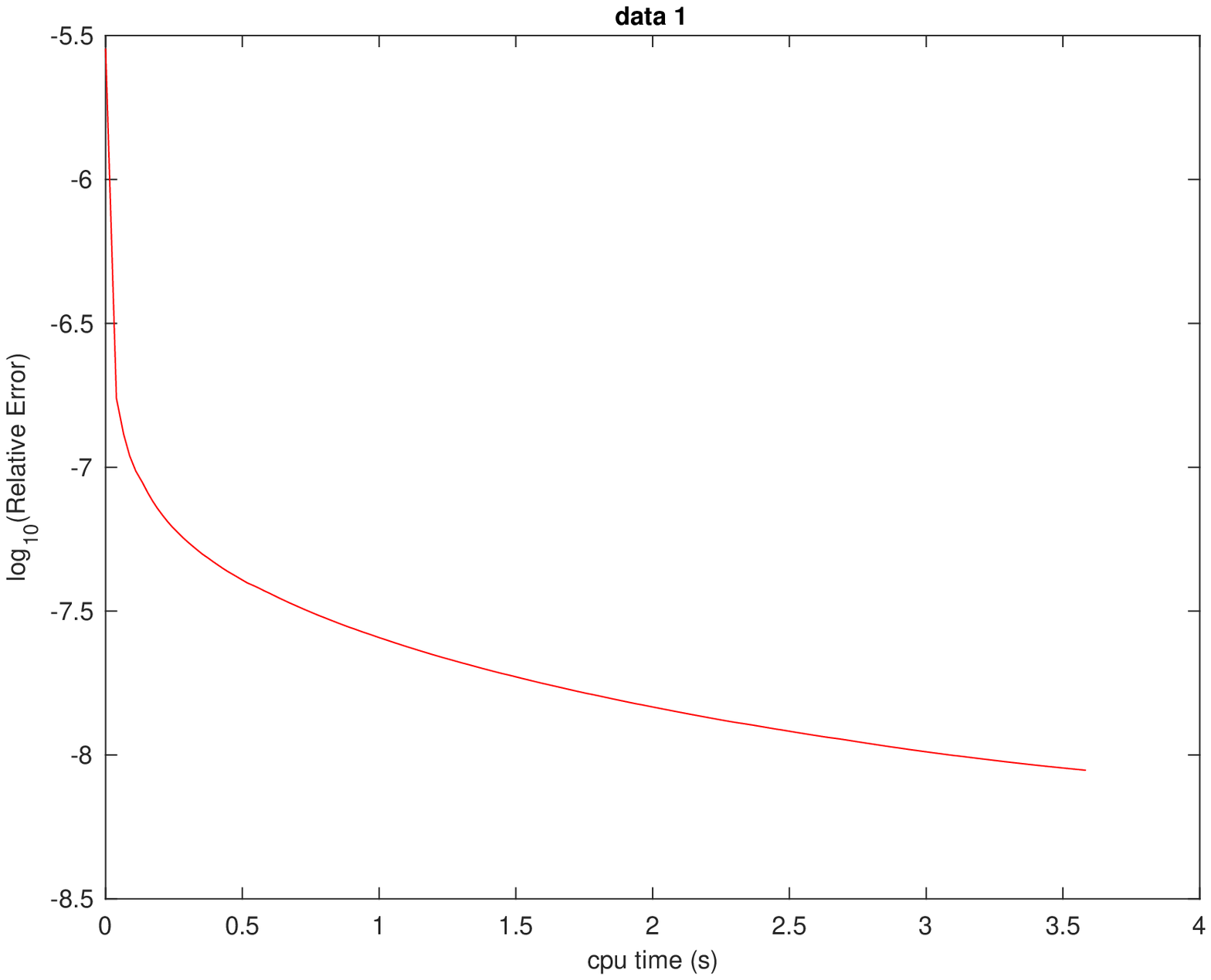}}&
{\includegraphics[width=.25\textwidth]{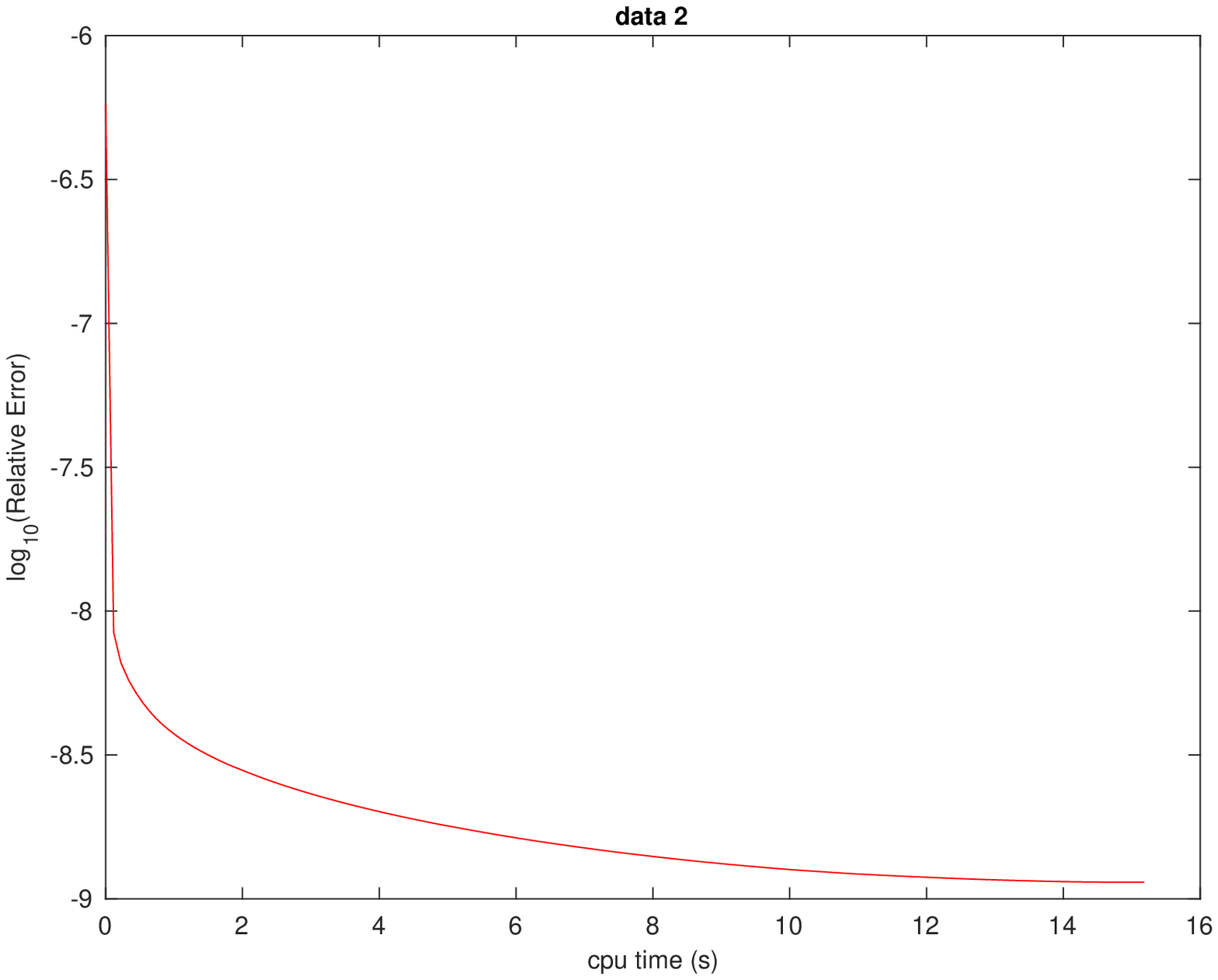}}&
{\includegraphics[width=.25\textwidth]{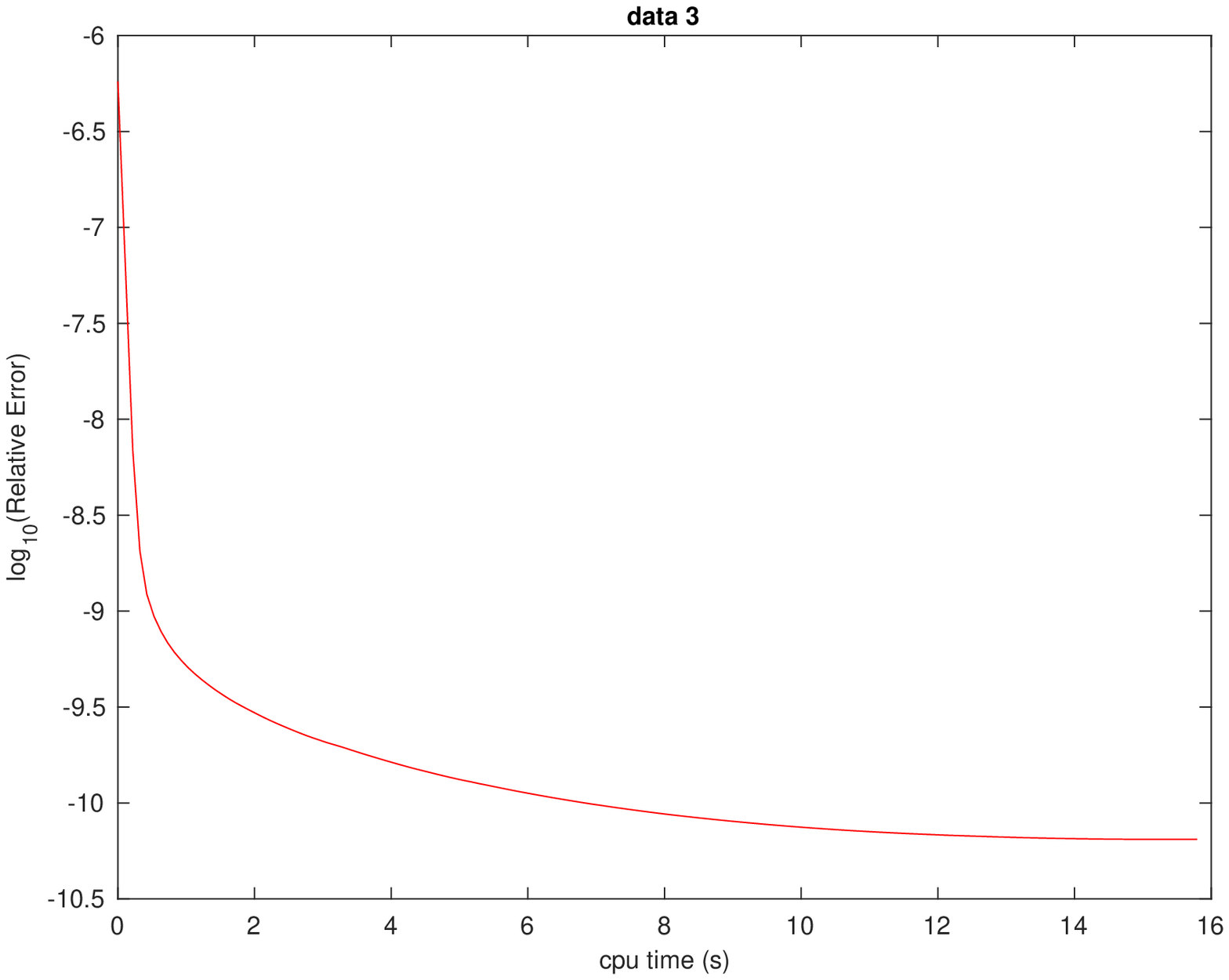}}\\
(d)&(e)&(f)\\
\end{tabular}
\end{center}
  \centering
  \begin{minipage}[b]{1.0\linewidth}
    \begin{center}
    \begin{tabular}{ccccc} 
\toprule
Given image  & SNR & PSNR & RelErr& CPU \\\midrule
(a) &26.38&40.99& 3.1816e-04&3.62\\
(b) &36.05 &48.47&1.3827e-04&18.96\\
(c) &35.83&58.34&4.7899e-05&19.22\\
 \bottomrule
    \end{tabular}
    \end{center}
      \end{minipage}
\caption {[log-sum results]
The results obtained by the CPALM method for 
solving  the log-sum penalization model (\ref{logsum-prob}). 
(a), (b), and (c) demonstrates the reconstruction results. 
We observe that SNR and PSNR improved significantly.
The image relative Error versus CPU (sec.) time are also shown  in 
(d), (e), (f). We observe that the trajectories are decreasing.}
 \label{Fig:logsum} 
\end{figure}
\begin{figure}
 \begin{center}
\begin{tabular}{ccc} 
{\includegraphics[width=.2\textwidth]{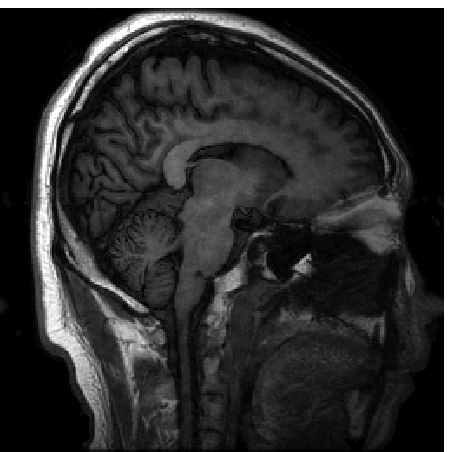}}&
{\includegraphics[width=.2\textwidth]{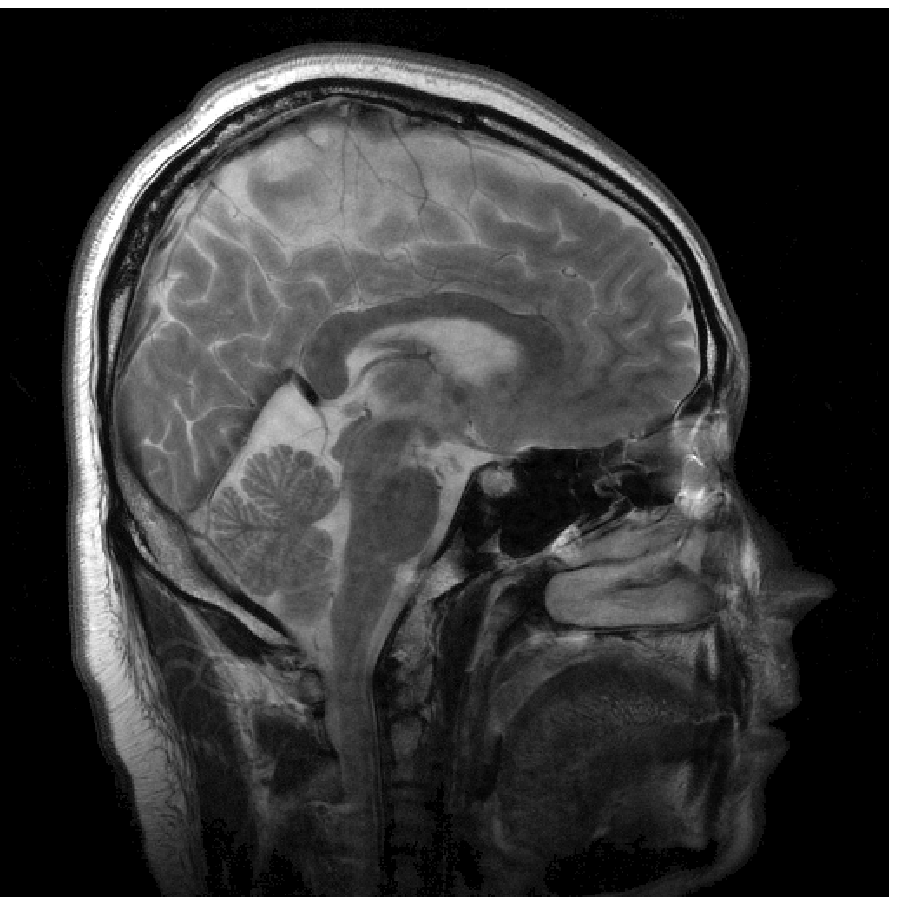}}&
{\includegraphics[width=.2\textwidth]{Figs/log-Rec-d3.eps}}
 \\
(a) Restored data1 & (b) Restored data2& (c) Restored data3\\
{\includegraphics[width=.25\textwidth]{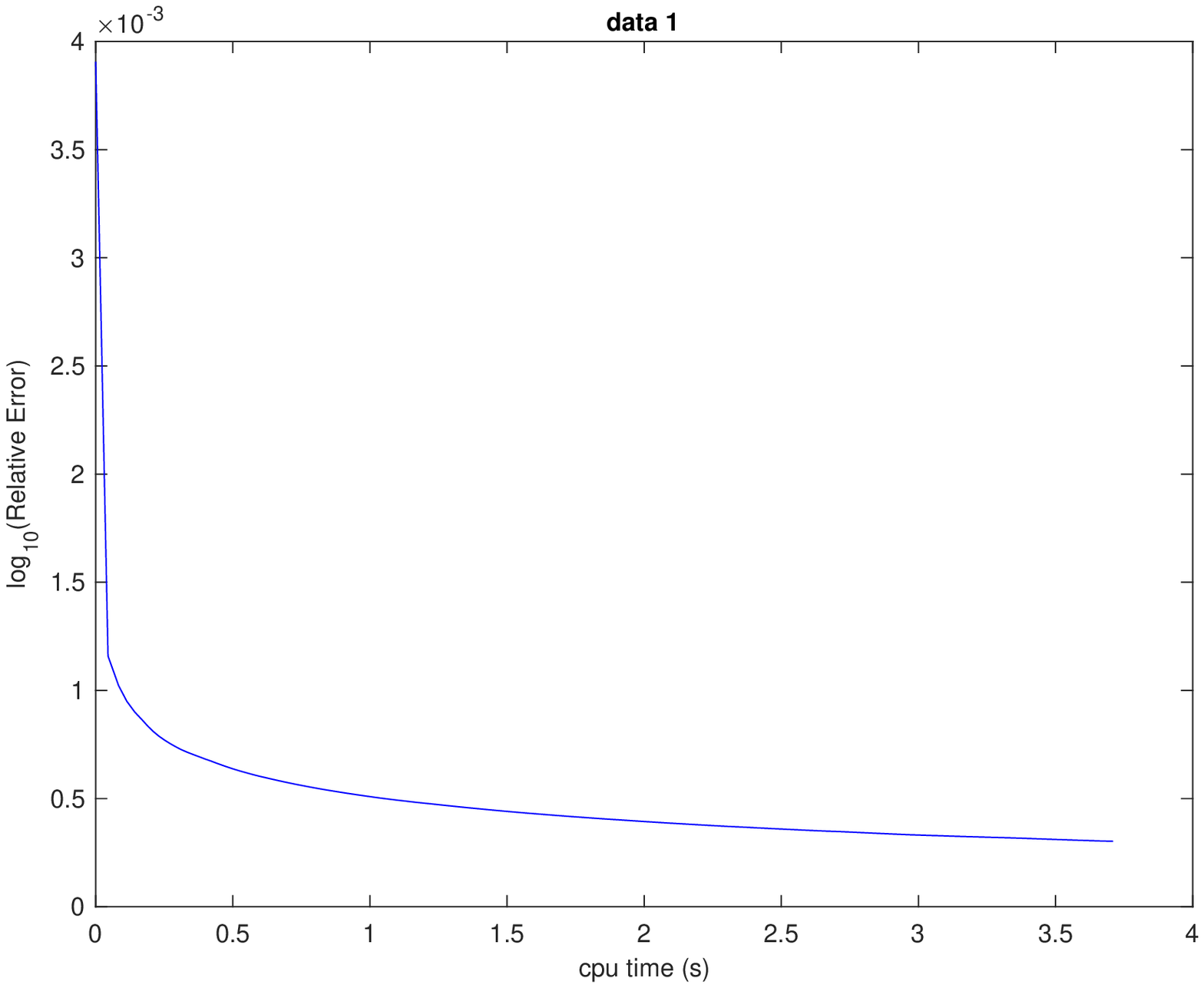}} &
{\includegraphics[width=.25\textwidth]{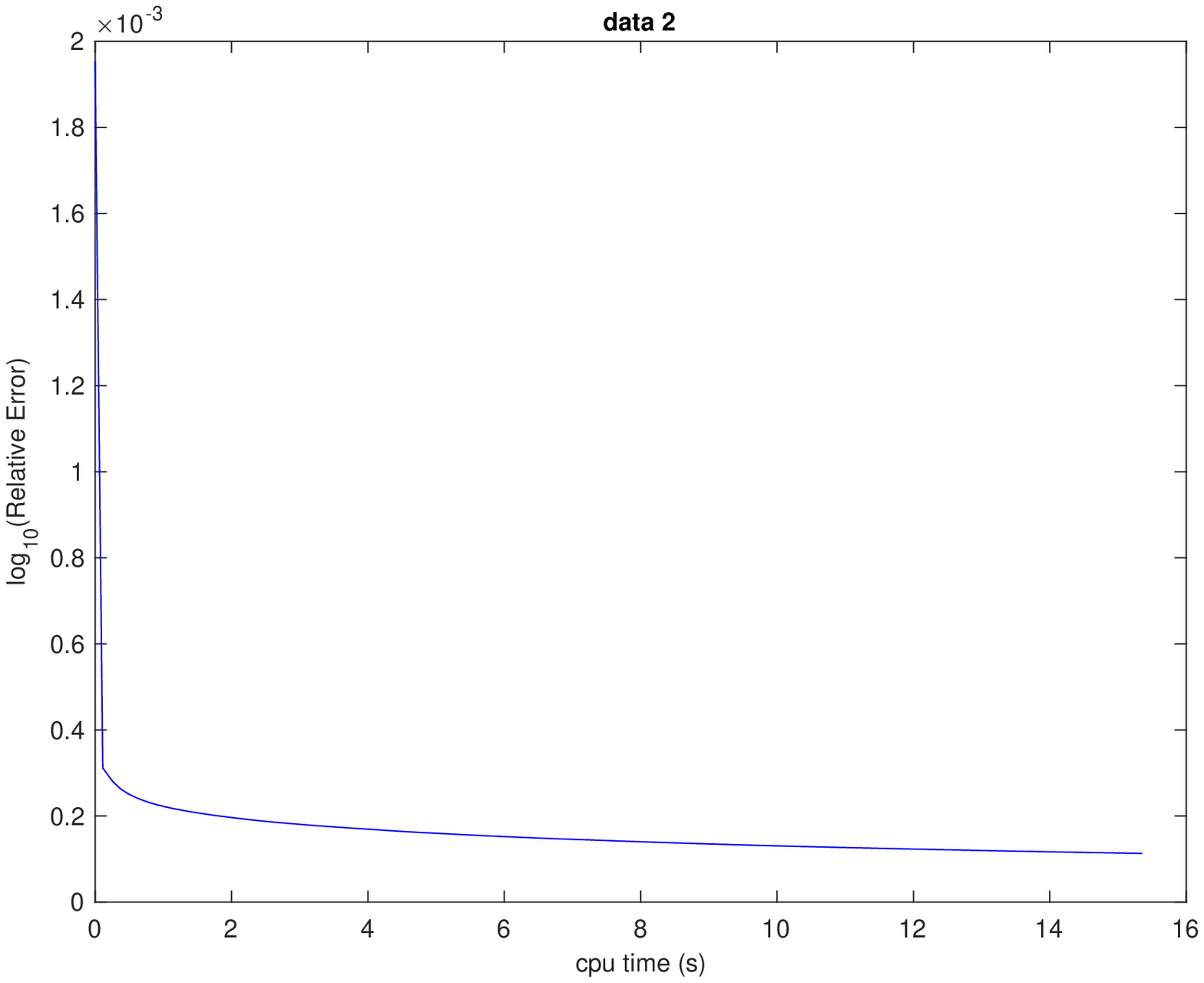}} &
{\includegraphics[width=.25\textwidth]{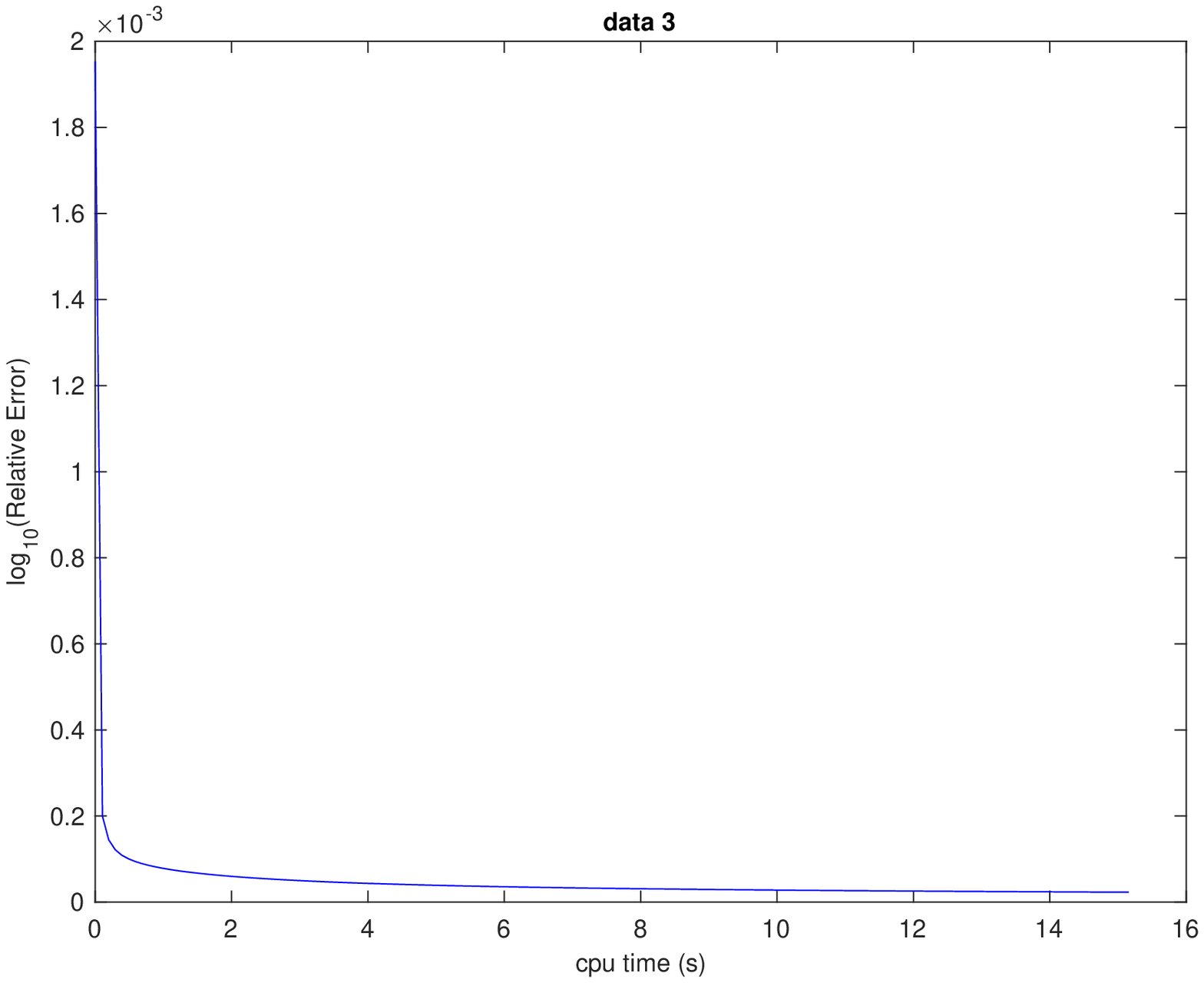}} \\
(d)&(e)&(f)\\
\end{tabular}
\end{center}
  \centering
  \begin{minipage}[b]{1.0\linewidth}
    \begin{center}
    \begin{tabular}{ccccc} 
\toprule
Given image  & SNR & PSNR & RelErr& CPU \\\midrule
(a) &26.26&40.73& 3.0236e-04&3.71\\
(b) &30.24&42.60&1.0452e-04&18.44\\
(c) &42.13&64.67&2.0681e-05&18.96\\
 \bottomrule
    \end{tabular}
    \end{center}
      \end{minipage}
\caption {[$\ell_p$, $p=0.5$ results]
The results obtained by the CPALM method for 
solving  the log-sum penalization model (\ref{lp-prob}). 
(a), (b), and (c) demonstrates the reconstruction results. 
We observe that SNR and PSNR improved significantly.
The image relative Error versus CPU (sec.) time are also shown  in 
(d), (e), (f). We observe that the trajectories are decreasing.}
\label{Fig:lp} 
\end{figure}
We can conclude that the proposed approach provides a good alternative 
to the class of PALM  method, in terms of both quality of reconstruction
and convergence speed.

\section{Concluding Remarks}\label{concludingremarks}
In this paper we proposed a new method, called CPALM, motivated by 
PALM \cite{PALM14}, to solve a class of composite 
nonconvex nonsmooth optimization problems.
The challenge arises from the fact that the nonsmooth term in the objective function
is a composition between a strictly increasing, concave,
differentiable function and a convex nonsmooth function.
To overcome the difficulty, we replace this term by an appropriate majorant function.
Theoretically, we proved that the CPALM method converges to a critical point 
of the problem and when the objective function satisfies the
KL property, the sequence generated by the CPALM method has a finite length.
Numerically, we applied the CPALM method to solve 
parallel MRI image reconstruction problems with nonconvex and nonsmooth log-sum as well as 
$\ell_p^p$ regularizations. The obtained results demonstrate the effectiveness of the proposed method.
In the future, we are interested in a accelerated version of CPALM by exploiting
the so-called heavy ball method of  Polyak \cite{polyak64}.

\bibliographystyle{siam}
\bibliography{library.bib}

\end{document}